\documentclass[11pt]{article}
\usepackage[a4paper, top=3cm, left=3cm, right=3cm, bottom=3.5cm]{geometry}
\usepackage{amsmath, amssymb}
\usepackage{enumerate}
\usepackage[T2A]{fontenc}
\usepackage{csquotes}
\usepackage[main=english, russian]{babel}

\usepackage[backend=biber, style=alphabetic, giveninits=true]{biblatex}
\renewbibmacro{in:}{}
\AtEveryBibitem{\clearfield{month}} 
\AtEveryBibitem{\clearfield{day}} 
\AtEveryBibitem{\clearfield{isbn}}
\AtEveryBibitem{\clearfield{issn}}
\AtEveryBibitem{\clearlist{language}}
\renewbibmacro*{doi+eprint+url}{%
    \printfield{doi}%
    \newunit\newblock%
    \iftoggle{bbx:eprint}{%
        \usebibmacro{eprint}%
    }{}%
    \newunit\newblock%
    \iffieldundef{doi}{%
        \usebibmacro{url+urldate}}%
        {}%
    }

\usepackage{mathtools}
\usepackage[dvipsnames]{xcolor}
\usepackage{amsthm}
\usepackage{tikz}
\usepackage{asymptote}
\usepackage[plainpages=false,pdfpagelabels,colorlinks=true,citecolor=blue,hypertexnames=false]{hyperref}
\usepackage[labelfont={sc}]{caption}
\usepackage{todonotes}
\usepackage[capitalize]{cleveref}
\usepackage{colonequals}
\usepackage{array}
\usepackage{comment,todonotes}
\presetkeys{todonotes}{inline}{}

\usepackage{algorithm}
\usepackage[noend]{algpseudocode}

\newcommand{\N}{\mathbb{N}}
\newcommand{\R}{\mathbb{R}}
\newcommand{\CC}{\mathbb{C}}
\newcommand{\Z}{\mathbb{Z}}
\newcommand{\Q}{\mathbb{Q}}
\newcommand{\Qb}{\overline{\mathbb{Q}}}
\newcommand{\Fpb}{\overline{\mathbb{F}}_p}
\newcommand{\F}{\mathbb{F}}

\newcommand{\ps}[1]{[\![#1]\!]}
\newcommand{\doesnotdivide}{\not\hspace{2.5pt}\mid}

\DeclareMathOperator{\den}{den}

\DeclareMathOperator{\res}{res}
\DeclareMathOperator{\lcm}{lcm}
\DeclareMathOperator{\Norm}{Norm}

\theoremstyle{definition}
\newtheorem{defi}{Definition}[section]
\newtheorem{ex}[defi]{Example}

\theoremstyle{plain}
\newtheorem{thm}[defi]{Theorem}
\newtheorem{lem}[defi]{Lemma}
\newtheorem{cor}[defi]{Corollary}
\newtheorem{prop}[defi]{Proposition}

\theoremstyle{remark}
\newtheorem{rem}[defi]{Remark} 

\title{An Effective Version of 
Kronecker's Theorem on the Splitting of Polynomials and the $p$-Curvature Conjecture for First-Order Differential Equations}

\author{Florian Fürnsinn and Lucas Pannier}
\begin{document}
	\maketitle
	\begin{abstract}
		We develop an effective version of Kronecker's Theorem on the splitting of polynomials, based on asymptotic arguments proposed by the Chudnovsky brothers, coming from Hermite-Padé approximation. In conjunction with Honda's proof of the $p$-curvature conjecture for first-order equations with polynomial coefficients we use this to deduce an effective version of the Grothendieck $p$-curvature conjecture for first-order equations. More precisely, we bound the number of primes for which the $p$-curvature of a given differential equation has to vanish in terms of the height and the degree of the coefficients, in order to conclude it has a non-zero algebraic solution. Using this approach, we describe an algorithm that decides algebraicity of solutions of first-order differential equations using $p$-curvatures, and report on an implementation in SageMath.
	\end{abstract}    {\let\thefootnote\relax}
	\section{Introduction} \label{sec:intro}

    A power series $f(x)\in \Q\ps{x}$ is called \emph{algebraic} if there exists a non-zero polynomial $P(x,y) \in \Q[x, y]$, such that $P(x, f(x))=0$. It is a well-known fact, already known to Abel, that every algebraic power series satisfies a linear differential equation
	\begin{equation} \label{eq:deq}
		a_n(x)y^{(n)}(x)+a_{n-1}(x)y^{(n-1)}(x)+\ldots +  a_1(x)y'(x)+a_0(x)y(x)=0
	\end{equation}
	with polynomial coefficients $a_i(x)\in \Q[x]$ for $i=0,\ldots, n.$ 
    
    Conversely, deciding whether the solutions of \eqref{eq:deq} are algebraic is an old problem dating back at least to Fuchs, and Liouville.  More precisely, given~\eqref{eq:deq}, one might ask, if 
    \begin{itemize}
    \itemsep0em
        \item[(A)] \emph{all}  solutions are algebraic,
        \item[(E)] there \emph{exists} at least one non-zero solution that is algebraic, or
        \item[(P)] a \emph{particular} solution, given for instance by enough coefficients of its power series expansion such that it is uniquely determined, is algebraic.
    \end{itemize}
    Notably, in his paper from 1980 that popularized the study of \emph{D-finite} series, i.e., power series that satisfy non-zero differential equations of the form \eqref{eq:deq}, Stanley asked for an ``algorithm suitable for computer implementation'' for problem (P) \cite{Sta80}.

    Problem (A) was solved algorithmically in a seminal paper by Singer \cite{Sin79}, relying on an algorithmic answer of Risch \cite{Ris70} in the case of first order differential equations with algebraic power series coefficients instead of polynomials -- this instance of the problem is known as Abel's problem -- and on work conducted by Painlevé and Boulanger \cite{Pai87, Bou97}. However, this answer is not fully satisfactory, as the complexity of the algorithm involves exponential bounds, making it in general not suitable for implementation. Recent progress on Stanley's question about problem (P) was made by Bostan, Salvy and Singer \cite{BSS25}, giving (semi-) algorithms that rely on the minimization of differential operators annihilating a given D-finite series.
	
	A different -- arithmetic -- approach to problem (A) was proposed by Grothendieck, in his \emph{$p$-curvature conjecture} \cite{Kat72}. Given \eqref{eq:deq}, we can pass to a matrix differential equation $Y'(x)=A(x)Y(x)$ of order $1$, with $A(x)\in \Q(x)^{n\times n}$ being the \emph{companion matrix} of \eqref{eq:deq}. Let $p$ be a prime number. The $p$-curvature of the equation is defined as the $\F_p(x)$-linear map $\psi_A: \F_p(x)^n\to \F_p(x)^n, Y(x)\mapsto (\partial - A_p)^p Y(x)$, where $A_p(x)\in \F_p(x)^{n\times n}$ denotes the reduction of $A(x)$ modulo $p$.  The $p$-curvature conjecture then states that \eqref{eq:deq} admits a basis of algebraic solutions if and only if its $p$-curvature vanishes for almost all, i.e., all but finitely many, prime numbers $p$. Cartier's Lemma (as attributed by Katz~\cite{Kat72}), asserts that the vanishing of the $p$-curvature is equivalent to the existence of a full basis of (algebraic) solutions of the reduction of the differential equation modulo $p$ in $\F_p\ps x$. Thus, the $p$-curvature conjecture can be seen as a local-global principle about the existence of algebraic solutions of a linear differential equation. 
	
	While in general wide open, the $p$-curvature conjecture is solved in many cases. Most notably, for ``equations coming from geometry'', i.e., suitable factors of Picard-Fuchs differential operators, a proof was given by Katz \cite{Kat72}, and for hypergeometric differential equations an elementary proof is a by-product of the classification of algebraic hypergeometric functions \cite{BH89, FY24}. For first order equations 
	\begin{equation}\label{eq:deq1}
		y'(x)=u(x) y(x)
	\end{equation}
	with $u(x) \in \Q(x)$ the three problems (A), (E) and (P) are equivalent, as there is only a one dimensional solution space of the equation. Honda \cite{Hon81} provided an elementary proof of the $p$-curvature conjecture in this case, by showing that the problem is equivalent to a number theoretic result by Kronecker \cite{Kro80}, which nowadays is often seen as a consequence of Chebotarev's Density Theorem~\cite{Cheb23, Cheb23a, Tsc26}.

	\begin{thm}[Kronecker]\label{thm:Kronecker}
		Let $R(x)\in \Q[x]$ be an irreducible polynomial. If for almost all prime numbers $p$, the reduction of $R(x)$ modulo $p$ has a root in $\F_p$, then $R(x)$ has a root in $\Q$, hence $R(x)$ is linear.
	\end{thm}
	
	For first order equations with \emph{algebraic} power series coefficients, i.e., for the equation $y'(x)=u(x) y(x)$ with $u(x)\in \Qb\ps{x}\cap \overline{\Q(x)}$, a proof of the $p$-curvature conjecture was given by the Chudnovsky brothers \cite{ChCh85}. They use \emph{Hermite-Padé approximation} to obtain their result. As a motivating example, they apply their methods to provide a new proof of Kronecker's Theorem, not relying on Chebotarev's Density Theorem, and in doing so giving also a new proof of Honda's Theorem. \medskip
	
	The $p$-curvature conjecture, as stated here, does not provide a way to \emph{decide} whether \eqref{eq:deq} has a basis of algebraic solutions, as it translates the statement to a statement about almost all, in particular, a statement about \emph{infinitely many}, prime numbers. In 1982 Katz extended the $p$-curvature conjecture to a conjecture about the Lie algebra of the differential Galois group of a differential equation~\cite{Kat82}. From his work presented in~\cite[\S 9]{Kat82} it follows by a noetherianity argument that - in the cases in which his conjecture is true, among them first-order equations with rational function coefficients - there exists a finite set of prime numbers, such that the vanishing of the $p$-curvatures modulo these primes implies algebraicity of the solutions. However, his argument is non-constructive. The purpose of this article is to explicitly construct a set of primes, for which it suffices to check the vanishing of the $p$-curvatures of a first order equation~\eqref{eq:deq1} -- the cardinality of the set depending on a suitable measure of the ``size'' of $u(x)$ -- to conclude that it has a non-zero algebraic solution. Our first main result is the following effective version of Kronecker’s Theorem.

    \begin{thm}\label{thm:intro1}
        Let \[R(w) =r_nw^n+\ldots + r_1w+r_0\in \Z[w]\] be a polynomial with leading coefficient $r_n > 0$, let $\Delta_R \coloneqq \res_w( R_0(w), R_0'(w))$, where $R _0(w)$ denotes the square-free part of $R(w)$, and suppose that the maximal modulus of its complex roots is bounded by $B\in\R$.
        Let $\delta(\Delta_R) \coloneqq \prod_{p|\Delta_R} p^{1/(p-1)}$ and set $M\coloneqq \left \lceil 2.826 \cdot r_n^3 \cdot \delta(\Delta_R)^3
        \right \rceil$ and $N\coloneqq \lceil 6.076\cdot BM\rceil$. Then $R(w)$ splits into linear factors in $\Q[w]$ if and only if its reduction modulo any prime number $p$ not dividing $r_n$ and less than $\sigma\coloneqq (2M+1)N+2M$ splits into linear factors in $\F_p[w]$.
    \end{thm}

    Our proof of this result makes the Chudnovskys' proof of Kronecker's result more explicit. The method of Hermite-Padé approximation inherently is effective in some sense, and in the conclusion of their article~\cite{ChCh85}, the Chudnovsky brothers also claim that their results are effective. However most of their computations are sketched without many details, preventing us from directly deriving explicit bounds. Our contribution is to confirm their claim and to work out concrete bounds.
    
    With this we obtain as a corollary the following method of deciding algebraicity of solutions of first-order differential equations.

	\begin{thm} \label{thm:intro2}
		Let $u(x)=c\cdot \frac{a(x)}{b(x)}\in \Q(x)$ be a rational function with $a(x), b(x)\in \Z[x]$ primitive\footnote{Recall that a polynomial $a(x)=a_nx^n + \ldots + a_1 x + a_0\in \Z[x]$ is called primitive, if $\gcd\{a_0,\ldots, a_n\}=1$.}, and $c\in \Q$. Assume that the coefficients of $a(x)$ and $b(x)$ are bounded in absolute value by $H$. Let 
        \[R(w) \coloneqq \res_x(b(x),a(x)-w\cdot b'(x))=r_nw^n+\ldots + r_1w+r_0\] and let $\Delta_b\coloneqq \vert \res_x(b(x), b'(x))\vert=\vert r_n\vert$. 
        If $b(x)$ is square-free and $\deg b(x) > \deg a(x)$ then, with the notation as in Theorem~\ref{thm:intro1}, the equation $y'(x)=u(x) y(x)$ has a non-zero algebraic solution if and only if its $p$-curvatures vanish for all primes not dividing $\Delta_b$ and smaller than $\sigma$. Otherwise, no non-zero algebraic solution exists.

        The computational complexity of checking that sufficiently many $p$-curvatures vanish is $\tilde O(\Delta_b^6 B)= \tilde O (H^{12n-6}n^{12n}3^{-3n})$ where the notation $\tilde O$ hides factors that are polynomial in $n$ and logarithmic in $H$.
	\end{thm}

    We also explain how the problem of deciding algebraicity of solutions of first-order differential equations with \emph{algebraic} coefficients reduces to the case of rational coefficients as in the theorem above.
   
    Moreover, we treat algorithmic aspects of the result and report on an implementation. We rely on a fast algorithm by Bostan and Schost~\cite{BS09} to quickly compute $p$-curvatures of first-order equations. The algorithm presented in this text does not outperform other known algorithms for certifying that a first-order differential equation has algebraic solutions, however in generic cases the algorithm detects quickly the presence of a transcendental solution, especially when the degree and size of the coefficients explode. We expect that any significant improvement in our approach for certifying the algebraicity of solutions would come from new theoretical results rather than algorithmic optimization. Although theoretically the algorithm could treat the case of differential equations with algebraic non-rational coefficients, we did not optimize it to treat this case efficiently.

    Our algorithm provides an instance of deciding algebraicity of D-finite functions using exclusively an arithmetic criterion. One could hope for results, similar in spirit, for first-order differential equations with algebraic coefficients, and for higher order equations in which the $p$-curvature conjecture is proven. The computation of $p$-curvatures for arbitrary order differential equations and deciding their nullity is possible and algorithms exist to perform these computations efficiently \cite{BCS15, BCS16}.

    Proving transcendence of solutions using our algorithm does not rely on our main results, Theorems~\ref{thm:intro1} and~\ref{thm:intro2}, while proving algebraicity using these theorems is highly inefficient. Thus we stress that the theoretical results of this paper, in particular the effective version of Kronecker's Theorem, should be seen as its main contribution.

    \paragraph{Structure of the Paper.}
    In Section~\ref{sec:Reduction} we revisit Honda's proof of the $p$-curvature conjecture for first-order differential equations by investigating its equivalence with Kronecker's Theorem in view of effective aspects. Relying on Theorem~\ref{thm:intro1}, we give a proof of Theorem~\ref{thm:intro2}, with the exception of the stated computational complexity. In Section~\ref{sec:comparison} we discuss different effective approaches to the $p$-curvature conjecture for first-order equations. The proof of Theorem~\ref{thm:intro1} takes up the entirety of Section~\ref{sec:effectiveKronecker}, in which we adapt the Chudnovskys' proof of Kronecker's Theorem. Afterwards, in Section~\ref{sec:algo}, we describe an algorithm that decides the algebraicity of the solution of first-order differential equation using the computation of $p$-curvatures. Its complexity estimate, worked out in Proposition~\ref{prop:complexity} finishes the proof of Theorem~\ref{thm:intro2}. Finally, in the last section, Section~\ref{sec:implementations}, we discuss our implementation of this algorithm in SageMath.

    \paragraph{Acknowledgments.} We are indebted to Alin Bostan, for his continuous support, for lots of insightful discussions and for many useful references, and to Rapha\"el Pagès for pointing out a serious flaw in Lemma~\ref{lem:binomred} in an earlier version of this text, and for his help with repairing it. We thank the anonymous reviewers, whose comments greatly improved the quality of this article.
    Moreover, we thank Markus Reibnegger for his interest in the work, for pointing out crucial references, and for fruitful discussions.

    The first-named author was funded by a DOC Fellowship (27150) of the \href{https://www.oeaw.ac.at/en/}{Austrian Academy of Sciences (ÖAW)} at the University of Vienna. The second-named author was funded by a CNRS MITI grant.
    Both authors were supported by the French–Austrian project EAGLES (ANR-22-CE91-0007 \& FWF grant \href{https://doi.org/10.55776/I6130}{10.55776/I6130}).
	
	\section{Reformulations of the Problem}\label{sec:Reduction}

    This section contains no new results. Its purpose is to describe the equivalence between the $p$-curvature conjecture for the first-order equation \eqref{eq:deq1} and Kronecker's Theorem~\ref{thm:Kronecker}. The results can be found in one form or another in the literature \cite{Hon81, ChCh85, vdP96, BCR24}. 
    
    % We could avoid dealing with $p$-curvatures in this section by shifting our focus to algebraic solutions of differential equations using Cartier's Lemma and see $p$-curvatures purely as a computational tool. We decide not to take this rout and to also discuss theoretical $p$-curvatures to fit in the classical framework of the $p$-curvature conjecture. 
    
	We consider the equation \eqref{eq:deq1} with $u(x)\in\Q(x)$.	A nonzero solution $y(x)$ has the form $y(x)=\exp(\int u(x)\mathrm{d} x)$.
	This expression combines two operations that do not preserve rationality or algebraicity in different ways.
	A primitive of a rational function remains a rational function if and only if all its residues are zero, whereas a fraction $\frac{\alpha}{x-\beta}$ has a logarithm as its primitive.
	Contrarily, the exponential acts in such a way that for a nonzero algebraic function $f(x)$, the function $\exp(f(x))$ is transcendental. This can be seen as a particular case of a conjecture by Schanuel about the transcendence degree of a set of power series and their exponentials, that was proven by Ax \cite{Ax71}.

	Hence, for a function $y(x)=\exp(\int u(x)\mathrm{d} x)$ to be algebraic, $u(x)$ cannot have a polynomial part or poles of order more than one, so that $y(x)$ factors into a product of $(x-\beta)^\alpha$ where $\beta$ is a pole of $u(x)$, and $\alpha$ is the residue at $\beta$.
	For such $u(x)$, the function $y(x)$ is algebraic if and only if the residues of $u(x)$ are rational numbers.

    \begin{rem} \label{rem:xalpha}
        We used the fact that the function $x^\alpha$ is algebraic if and only if $\alpha$ is rational. Indeed, if $\alpha$ is rational, clearly $x^\alpha$ is algebraic. Assume conversely, that $x^\alpha$ is algebraic. Then there exists a non-zero relation of the form $\sum_{m,n} c_{m,n}x^{\alpha m + n}=0$. In particular, as the functions $x^\beta$ for $\beta\in \CC$ are $\CC$-linearly independent, there have to exist $m_1, m_2, n_1, n_2\in \N$, such that $x^{\alpha m_1 + n_1} = x^{\alpha m_2 + n_2}$, i.e., $\alpha = (n_2-n_1)/(m_1-m_2)$ is rational.
    \end{rem}

    A convenient way to compute the residues of the rational function $u(x)$ is provided by \emph{Rothstein-Trager resultant} \cite{Rot76, Rot77, Tra76}. We state their result here in a simplified, but for our purposes sufficient, form.

    \begin{thm}[Rothstein, Trager] \label{thm:RTres}
        Let $u(x)=a(x)/b(x)$ be a rational function with only simple poles and $\deg a(x)<\deg b(x)$. Then its residues are precisely the roots of \emph{the Rothstein-Trager resultant}
        \[R(w)\coloneqq \res_x(b(x), a(x)-w\cdot b'(x)).\]
    \end{thm}

     We investigate complexity estimates for the Rothstein-Trager resultant in Section~\ref{ssec:RT}. Altogether we get the following equivalent characterizations.

    \begin{prop} \label{prop:TFAE0}
        The following statements for the differential equation~\eqref{eq:deq1} are equivalent.
        \begin{enumerate}[(1)]
        \itemsep0em 
            \item All solutions of~\eqref{eq:deq1} are algebraic.
            \item The coefficient $u(x)$ can be written as
            \begin{equation}\label{eq:1poles}
			u(x) = \sum\limits_{i=1}^n \frac{\alpha_i}{x-\beta_i}
		\end{equation}
		with distinct algebraic poles $\beta_i\in\Qb$ and rational residues $\alpha_i\in \Q$.
        \item We have $\deg a(x)<\deg b(x)$, all poles of $u(x)$ are simple and its Rothstein-Trager resultant factors into linear factors in $\Q[w]$.
        \end{enumerate}
    \end{prop}

    The conditions on the degree and the poles of $u(x)$ in the third item are equivalent to the differential equation \eqref{eq:deq1} being Fuchsian.
    
    Similar considerations hold in characteristic~$p$. Moreover, by Cartier's Lemma, the existence of a non-zero solution of a first-order differential equation $y'(x)=u(x) y(x)$ is equivalent to the vanishing of its $p$-curvature\footnote{The existence of a non-zero polynomial solution in $\F_p[x]$, a non-zero algebraic solution in $\F_p\ps{x}$ or \emph{any} non-zero solution in $\F_p\ps{x}$ are equivalent \cite[Lem.~1]{Hon81}.}. In the particular case of first-order equations, the $p$-curvature is given by a one-dimensional matrix, whose entry we also call, by abuse of notation, the $p$-curvature. There is an explicit formula for this $p$-curvature in terms of the coefficient $u(x)$, that already implicitly appears in Jacobson's work~\cite{Jac37}: it is given by $u(x)^p+u^{(p-1)}(x)$.

    Putting all these observations together, we infer the following results characterizing the existence of (algebraic) solutions in characteristic $p$.    
    
    \begin{prop} \label{prop:TFAEp}
        Let $p$ be a prime number. The following statements are equivalent.
        \begin{enumerate}[(1)]
            \itemsep0em
            \item The $p$-curvature $u(x)^p+u^{(p-1)}(x)$ of \eqref{eq:deq1} vanishes.
            \item The reduction of \eqref{eq:deq1} has an algebraic solution in $\F_p\ps{x}$.
            \item The coefficient $u(x)$ can be written as
            \[u(x)=\frac{a(x)}{b(x)}=\sum_{i=1}^n \frac {\alpha_i}{x-\beta_i},\]
            with $\beta_i\in \Fpb$ and $\alpha_i\in \F_p$.
            \item We have $\deg a(x)<\deg b(x)$, all poles of $u(x)$ are simple and its Rothstein-Trager resultant factors into linear factors in $\F_p[w]$.
        \end{enumerate}
    \end{prop}

    The equivalence of (1) and (2) is Cartier's Lemma (see \cite[Theorem 3.19]{BCR24} for a precise statement and proof). Statements (3) and (4) are equivalent by Theorem \ref{thm:RTres}. Assuming (3) we can write down an explicit algebraic solution, hence (3) implies (2). Finally, by Jacobson's formula, the vanishing of the $p$-curvature of Equation~\eqref{eq:deq1} implies that the partial fraction decomposition of $u(x)$ has precisely the shape given in (3).

    Honda's proof of the $p$-curvature conjecture for~\eqref{eq:deq1} essentially consists in recognizing that Kronecker's Theorem implies that if (4) in Proposition~\ref{prop:TFAEp} holds for almost all prime numbers $p$, then (3) in Proposition~\ref{prop:TFAE0} holds as well. Our following considerations throughout this article are focused on the question, whether (4) in Proposition~\ref{prop:TFAEp} for \emph{finitely} many prime numbers suffices to conclude (3) in Proposition~\ref{prop:TFAE0}.

    Note that Honda actually proved more, namely that the $p$-curvature conjecture for first-order equations is \emph{equivalent} to Kronecker's Theorem. We refer the reader for the other implication to Honda's original article~\cite{Hon81}.

    Let us next investigate which prime numbers are exceptional in the sense that there is no solution of the reduction of the differential equation $y'(x)=u(x) y(x)$ with $u(x)=a(x)/b(x)$ in $\F_p\ps x,$ despite the existence of a solution in characteristic $0$. In other words we are investigating for which prime numbers assertion (4) in Proposition~\ref{prop:TFAEp} might be false, although (3) in Proposition~\ref{prop:TFAE0} holds. Trivially, if a polynomial splits in $\Z[x]$ into linear factors, it splits into linear factors in $\F_p[x]$ not just for almost all prime numbers, but for all prime numbers. However, it might happen that the degree of $b(x)$ drops, when reduced modulo $p$, or that two distinct poles of $u(x)$ of order $1$ collapse to a pole of higher order, when reduced modulo $p$. Luckily, both of these exceptions can be easily read off the leading coefficient (up to sign)  $\Delta_b=|\res(b(x), -b'(x))|$ of the Rothstein-Trager resultant, as already noticed by van der Put, c.f.~\cite[Prop.~3.2(2)]{vdP96}.

    \begin{prop} \label{prop:Put}
        Assume that $y'(x)=u(x) y(x)$ with $u(x)=a(x)/b(x)$ has an algebraic solution. If $p\doesnotdivide \Delta_b$ then the $p$-curvature of the equation vanishes. 
    \end{prop}

    The following example illustrates both types of exceptional primes.
    \begin{ex}
        Consider the rational function 
        \[u(x)= \frac {a(x)}{b(x)} = \frac {x+2} {2x^2+x-1}= \frac 5 6 \cdot \frac 1 {x-\frac 1 2} - \frac 1 3 \cdot \frac 1 {x+1}\]
    and the associated differential equation $y'(x)=u(x) y(x)$. It has the algebraic solution $y(x)= \frac{\left(2 x-1\right)^{{5}/{6}}}{\left(x+1\right)^{{1}/{3}}}$, which cannot be reduced modulo $2$ and $3$, but modulo all other prime numbers. The resultant $\res(b(x), -b'(x))$ is equal to $-18$. For $p=2$ the denominator of $u(x)$ reduces to a degree one polynomial. For $p=3$ the two poles $x=\frac 1 2$ and $x=-1$ collapse to a single pole of order $2$.
    \end{ex}

    We are now ready to prove Theorem~\ref{thm:intro2}, assuming Theorem~\ref{thm:intro1}.

    \begin{proof}[Proof of Theorem~\ref{thm:intro2}]
        First, we can reduce without loss of generality to the case $c=1$. Indeed, if $y(x)$ satisfies $y'(x)=u(x) y(x)$, then $(y(x)^c)'=c\cdot u(x) y(x)^c$ and for a given $c\in \Q$, the function $y(x)$ is algebraic if and only if $y(x)^c$ is algebraic. 
        
        By Theorem~\ref{thm:intro1}, $R(w)$ factors into linear factors in $\Q[w]$ if and only if its reduction also splits completely in $\F_p[w]$ for all prime numbers $p$ up to $\sigma$, not dividing $\Delta_b$. Note that $\Delta_b$ is non-zero by the assumption that $b(x)$ is square-free. By Proposition~\ref{prop:TFAEp}, using that $\deg a(x)<\deg b(x)$, this in turn is equivalent to the vanishing of the $p$-curvatures for all prime numbers not dividing $\Delta_b$ and smaller than $\sigma$.

        The complexity estimates 
        follow from Proposition~\ref{prop:complexity}.
    \end{proof}

    \begin{rem}
        By Proposition~\ref{prop:TFAE0}, the equation having a non-zero algebraic solution is equivalent to $\deg b(x)>\deg a(x)$, that $b(x)$ only has simple poles, and that $R(w)= \res_x(b(x), a(x)-w\cdot b'(x)$ factors in linear factors in $\Q[w]$. We check here the first two conditions directly. However, if one of them is not met, it is also not met when reducing $u(x)$ modulo almost all primes, so by Proposition~\ref{prop:TFAEp} almost all $p$-curvatures will not vanish. Having said that, the bound on the primes in this case is not controlled solely by the Rothstein-Trager resultant, as the following example illustrates:
        Let $a(x) = Ax + 1, b(x) = x$ for $A\in \Q$. Then the Rothstein-Trager resultant is $R(w) = -w +1$ for every choice of $A$. Assume we use Theorem~\ref{thm:intro1} to obtain a bound $\sigma$ on the numbers of $p$-curvatures to check to decide algebraicity. If we now set $A=\sigma!$, the $p$-curvatures of our differential equation will vanish up to $\sigma$, but will not vanish for any prime larger than $\sigma$.
    \end{rem}

    Finally we note that our arguments can easily be extended to equations $y'(x)=u(x) y(x)$ with $u(x)\in \Qb[x]$ by using the following well-known result, that is a direct consequence of Vieta's formulas.

    \begin{prop} \label{prop:Qb}
        Let $K$ be a number field, let $R(w)\in K[w]$ with leading coefficient $r_n\in K\backslash\{0\}$.
        If $R(w)$ splits completely over $K$ with all its roots in $\Q$, then $\frac{1}{r_n}R(w)$ is a polynomial in $\Q[w]$.
    \end{prop}

    \section{Other Approaches} \label{sec:comparison}
    
    \subsection{Finding Rational Roots}
    \label{ssec:factoring}

    Once we have computed the Rothstein-Trager resultant $R(w)\in\Q[w]$ associated to $u(x)$, checking the rationality of the residues (and thus solving the problem of deciding algebraicity of the solutions of \eqref{eq:deq1}) amounts to finding the rational roots of $R(w)$.
    Then the residues of the rational function $u(x)$ are all rational if and only if its Rothstein-Trager resultant $R(w)$ splits completely over $\Q$. 

    A na\"{\i}ve approach to finding the rational roots of a polynomial is to factor it completely over $\Q$ and consider the factors of degree one.
    However this is more than what we are looking for and specific more efficient algorithms to find the rational roots of a polynomial exist. We refer the reader to \cite[\S 14-16]{Gz13} and \cite[\S21]{BCGLLSS17} for details about such algorithms and an analysis of their complexity.
    
    We obtain Algorithm~\ref{algo:ratroots}.
    
    \begin{algorithm}
        \caption{Deciding algebraicity by finding rational roots}\label{algo:ratroots}
        $\begin{array}{ll}
            \textbf{Input: }& \text{Polynomials } a(x), b(x)\in\Q[x]. \\ 
            \textbf{Output: }& \text{The nature \texttt{Algebraic} or \texttt{Transcendental} of solutions of } y'(x)=\frac{a(x)}{b(x)}y(x).
        \end{array}$
        \begin{algorithmic}[1]
        \State $R(w)\leftarrow \res_x(b(x),a(x)-w\cdot b'(x))$;
        \State Find rational roots of $R(w)$;
        \If{there are $\deg(R(w))$ rational roots} \Return \texttt{Algebraic};
        \Else{ \Return \texttt{Transcendental};}
        \EndIf
        \end{algorithmic}
    \end{algorithm}

    \begin{rem} \label{rem:Qbfact}
        Starting from a rational function in $u(x)\in\Qb[x]$, deciding the algebraicity of solutions of $y'(x)=u(x)y(x)$ reduces to the same problem.
        Following Proposition~\ref{prop:Qb}, once we have computed the Rothstein-Trager resultant, we can divide by its leading coefficient and if the obtained polynomial has non-rational coefficients we can conclude on the transcendence of the solutions of the input differential equation.
        Else we continue with the polynomial in $\Q[w]$ as described in this section.
    \end{rem}

    \subsection{The Least Prime That Does Not Split in a Number Field}
    The matter of bounding the least rational prime number that does not split completely in a given (Galois) number field has been extensively studied in the literature.
    It is a special case of the more general problem of finding an upper bound on the least prime that has a prescribed splitting pattern in a number field~\cite{LMO79, AK19, KW22}, in accordance with the Chebotarev Density Theorem. More precisely, given a Galois field extension $K/\Q$ of degree~$d$ and discriminant~$D$, there are several asymptotic estimates on the size of the smallest rational prime that does not split completely in $K$. An explicit statement was proven by Vaaler and Voloch \cite[Thm.~1]{VV00}.

    \begin{thm}[Vaaler, Voloch]
        Let $K/\Q$ be a number field of degree $d$ and discriminant~$D$. If $\exp(\max\{105, 25(\log (d))^2\})\leq 8 D^{\frac 1 {2(d-1)}}$ then there exists a prime $p$, such that $p$ does not split completely in $K$ and $p\leq 26d^2D^{\frac 1 {2(d-1)}}$.
    \end{thm}

    Sharper, asymptotic bounds were proven, both under the assumption of the Generalized Riemann Hypothesis by Murty~\cite{Mur94}, and unconditionally~\cite{Mur94, Li12}. \medskip

For a prime $p$ splitting completely in the splitting field of a polynomial $R(w)$ is equivalent to the reduction of $R(w)$ modulo $p$ splitting into linear factors in $\F_p[w]$, except for finitely many primes.

        If $L/\Q$ is the splitting field of $R(w)\in\Q[w]$, then $R(w)$ splits completely in $L$, and it splits completely in any of the residue fields of $L$. Let $p$ be a prime of $\Z$ that splits completely in $L$, then in particular the corresponding residue fields of $L$ are equal to their prime fields $\F_p$~\cite[\S1, Prop.~8.2, p.~46{;} p.~49]{Neu99}, so $R(w)$ splits completely over $\F_p$.

    The other implication can be shown by applying the Dedekind-Kummer Theorem \cite[\S1, Prop.~8.3]{Neu99}:
        We can reduce to the case where $R(w)$ is monic and in $\Z[w]$ (divide by its leading coefficient and consider $m^{\deg(R)}R(w/m)$ for $m\in\Z$ a common denominator of its coefficients).
        Consider an irreducible factor of $R(w)$ of degree at least $2$ with root $\alpha$. This factor will split completely modulo $p$ for almost all primes of $\Z$. Hence in $K_1=K(\alpha)$ the Dedekind-Kummer Theorem states that almost all primes of $\Z$ split completely. We then repeat the argument with the root of an irreducible factor of $R(w)$ of degree at least $2$ over $K_1$ and so on until reaching $K_r = L$. Each step excludes finitely many primes, thus almost all primes split completely in $L$.
        \medskip
    
    This gives another approach to finding an effective version of Kronecker's Theorem. There seem to exist obstacles, however.
    
    First, most results in this direction are of asymptotic nature, making it hard to convert them into an explicit bound. The exception is Vaaler and Voloch's statement given above. However, they have an assumption on the minimal size of $D$ that exceeds the computational capacity of modern computers. Li claims~\cite[p.~1062]{Li12} that this assumption is artificial and better results could be expected for small $D$.

    Second, one would need to estimate, given a polynomial $R(w)$, the discriminant and the degree of a potential splitting field of $R(w)$. Given that we cannot assume $R(w)$ to be irreducible, the discriminant of the polynomial and the discriminant of its splitting field are not related in an obvious way.

    While both of these problems seem to be manageable in principle, we are not aware of any algorithm for solving our problem based on this approach.

    \subsection{Indicial Equations}
    
    As discussed in Section~\ref{sec:Reduction}, the problem of  deciding algebraicity of the solutions of~\eqref{eq:deq1} can be solved by investigating its singularities, and, in particular by deciding whether the residues of the rational function coefficient $u(x)$ are rational. Working directly in the framework of differential equations instead, we have the following addition to Proposition~\ref{prop:TFAE0}.
    \begin{prop} \label{prop:TFAE0a}
        The following statements for the differential equation~\eqref{eq:deq1} are equivalent.
        \begin{enumerate}[(1)]
        \itemsep0em 
            \item All solutions of~\eqref{eq:deq1} are algebraic.
            \setcounter{enumi}{3}
            \item All singularities of~\eqref{eq:deq1} are regular singular and its local exponent at each of these singularities is rational.
        \end{enumerate}
    \end{prop}
    \begin{proof}
        The singularities of~\eqref{eq:deq1} are precisely the poles of $u(x)$, and possibly $\infty$. As the order of the differential equation is $1$, a singularity at $x_0$ is regular, if and only if the valuation of $u(x)$ at $x_0$ is at least $-1$, i.e., if the pole of $u(x)$ is of order at most $1$. The local exponent at $x_0$ is precisely given by the residue of $u(x)$ at $x_0$. Investigating the singularity at $\infty$ amounts to investigate the behavior of $-x^2 y'(x)= u(\frac 1 x)y(x)$ at $0$. So the singularity is present, if and only if $\deg a(x)\geq \deg b(x) -1$ and it is regular if and only if $\deg a(x)= \deg b(x)-1$. The local exponent at $\infty$ automatically is rational if $u(x)$ has rational coefficients. So we have shown that (4) is equivalent to assertion (2) of Proposition~\ref{prop:TFAE0}.
    \end{proof}

    To use this equivalent criterion to decide algebraicity of solutions, one can factor $b(x)$ to determine the singularities of the equation and then compute the local exponents by computing the \emph{indicial polynomial} at the zeroes of $b(x)$. This polynomial is of degree one with coefficients in the splitting field of $b(x)$. Its root is rational, if and only if it is a $\Qb$-multiple of a polynomial in $\Q[x]$.
    
    This approach leads to the following algorithm.
    \begin{algorithm}
        \caption{Deciding algebraicity with indicial equations}\label{algo:indeq}
        $\begin{array}{ll}
            \textbf{Input: }& \text{Polynomials } a(x), b(x)\in\Q[x]. \\ 
            \textbf{Output: }& \text{The nature \texttt{Algebraic} or \texttt{Transcendental} of solutions of } y'(x)=\frac{a(x)}{b(x)}y(x).
        \end{array}$
        \begin{algorithmic}[1]
        \State $b_1(x),\dots,b_r(x)\leftarrow \text{irreducible factors of }b(x) \text{ over } \Q$;
        \For{$i=1,\dots,r$} $z\leftarrow$ RootOf$(b_i(x))$;
        \State Compute the indicial equation at $x=z$;
        \If{the root of the indicial equation is not rational} \Return \texttt{Transcendental};
        \EndIf
        \EndFor
        \State \Return \texttt{Algebraic};
        \end{algorithmic}
    \end{algorithm}

	\section{An Effective Version of Kronecker's Theorem} \label{sec:effectiveKronecker}
    This section is devoted to the proof of Theorem~\ref{thm:intro1}. We outline the steps here. In \cite{ChCh85} D.V. and G.V. Chudnovsky provide an elementary proof of Kronecker's Theorem (and with it, according to the discussion in Section~\ref{sec:Reduction}, of Honda's proof of the $p$-curvature conjecture for first-order equations) using Hermite-Padé approximation.

    One main ingredient is the fact that the function $x^\alpha$, for example defined, setting $z\coloneqq x-1$, by the power series
    \begin{equation*}
            x^\alpha = \sum\limits_{n\geq 0} \binom{\alpha}{n} (x-1)^n = \sum\limits_{n\geq 0} \binom{\alpha}{n} z^n,
    \end{equation*}
    is algebraic over $\Q(x)$ if and only if $\alpha \in \Q$, see Remark~\ref{rem:xalpha}. We will assume by contradiction that a root $\alpha$ is irrational, although its annihilating polynomial for almost all prime numbers $p$ splits into linear factors.

    The next ingredient is Hermite-Padé approximation. The ideas of this method were first introduced by Hermite in \cite{Her74} to prove the transcendence of $e$, and then extended by him and Padé, his student \cite{Her93, Pad92, Pad94}. It produces for a finite number of power series, a list of polynomials of prescribed degrees, such that the sum of the power series, weighted by the polynomials, has highest possible order. More precisely, let $K$ be a field. Let $f_1(z),\ldots, f_m(z)\in K\ps{z}$ be a tuple of power series. Let $n_1,\ldots, n_m$ be non-negative integers and let $P_1(z),\ldots, P_m(z)\in K[z]$ be polynomials of degree $n_1,\ldots, n_m$ respectively. If the power series \[g(z)\coloneqq P_1(z)f_1(z)+\cdots + P_m(z)f_m(z)\in K\ps{z}\] has order greater than or equal to $n_1+\cdots + n_m + m -1$ at $0$, then the tuple $(P_1(z), \ldots, P_m(z))$ is called a \emph{Hermite-Padé approximant} to $( f_1(z),\ldots, f_m(z))$ of \emph{type} $(n_1,\ldots, n_m)$, and $g(z)$ is called the \emph{remainder} of the approximation. For further exposition of the concept, see \cite{Mah68, Jag1, Jag2}.
    
    In the following, a well-known formula for some Hermite-Padé approximants of the consecutive powers $(1-z)^{(i-1)\alpha}$ with $\alpha \not \in \Q$ for $1\leq i \leq 2M+1$, with coefficients of uniform degree $N$ for any integer values of $M$ and $N$ is exploited. It was introduced by Padé in \cite{Pad01} and studied and extended by Jager, Mahler, and G. Chudnovsky \cite{Jag2, Mah68, Chu83}, but could already be deduced from Hermite's work \cite{Her74}. The result also includes an explicit expression for the first non-vanishing term in this approximation.

    It is clear that this first non-vanishing term of this approximation will be an algebraic number. In the following one considers a certain multiple of this coefficient. One can estimate its norm, and its denominator by two different estimates, coming from either side of the equation for the Hermite-Padé approximation, depending on $M$ and $N$. The bound on the denominator uses the fact that if $\alpha$ reduces to an integer modulo a prime $p$, certain binomial coefficients involving $\alpha$ have non-negative $p$-adic valuation. From there, Chudnovsky and Chudnovsky use an asymptotic argument, that  for sufficiently large $M$ and $N$, the bounds contradict the trivial inequality that the norm of a non-zero algebraic integer is greater or equal to one.

    We will argue along the same lines, but we will not work asymptotically, but find upper bounds on the values of $M$ and $N$ for the contradiction to occur. In addition, we use that the argument of the Chudnovsky brothers relies on ruling out prime divisors of the denominators of a finite set of binomial coefficients, in whose denominators \emph{a priori} only prime factors to a bound $\sigma$ depending on $M$ and $N$ may appear. This allows us to conclude if we assume that $R(w)$ factors into linear factors modulo $p$ for all primes $p$ up to $\sigma$.
    \medskip

    For the rest of the section we fix an irrational number $\alpha$ with square-free annihilating polynomial $R(w)=r_n w^n  + \cdots + r_1 w + r_0$ with $r_n>0$. We set $\Delta_R\coloneq \res_w(R(w),R'(w))\neq 0$, we define $\delta(\Delta_R) \coloneqq \prod_{p|\Delta_R} p^{1/(p-1)},$ and we pick $B\in\R$ to be a bound on the maximal modulus of a root of $R(w)$. We let $L$ be the splitting field of $R$. The \emph{denominator} $\den(\gamma)$ of $\gamma\in L$ is the smallest positive integer $a\in \N_{>0}$, such that $a\gamma\in \mathcal{O}_L$, the ring of integers of $L$. In particular, we have $\den(\alpha)\vert r_n$. Let $M, N$ be two positive integers and define $\sigma = \sigma(M, N) \coloneqq (2M+1)N+2M$.  We will assume that $R(w)$ splits into linear factors for all primes up to $\sigma$, that do not divide $r_n$. The proof of Theorem~\ref{thm:intro1} is equivalent to showing that when choosing $M\coloneqq \left \lceil 2.826 \cdot r_n^3 \cdot \delta(\Delta_R) \right \rceil$ and $N\coloneqq 6.076\cdot BM$ the above assumptions lead to a contradiction. We will accomplish this in Proposition~\ref{prop:smaller1}, contradicting Corollary~\ref{cor:AnD}. \medskip

    \subsection{Hermite-Padé Approximants to Binomial Powers and Estimates on Their Remainder}

    An explicit formula for Hermite-Padé approximants to the consecutive powers $(1-z)^{(i-1)\alpha}$ for $1\leq i \leq 2M+1$, with coefficients of uniform degree $N$ over $L$ reads
    \begin{equation}\label{eq:Padé}
        \sum\limits_{i=1}^{2M+1} P_i(z)(1-z)^{(i-1)\alpha} = g(z) = \frac{N!^{2M+1}}{\sigma!}z^\sigma + O(z^{\sigma + 1}),
    \end{equation}
    where $O(z^{\sigma+1})$ is a power series in $z^{\sigma+1}\Q\ps{z}$ and where
    \begin{equation}\label{eq:pih}
        P_i(z) \coloneqq \sum\limits_{h=0}^N p_{i,h}z^h\quad \text{with} \quad p_{i,h} \coloneqq \binom{N}{h}\left(\prod\limits_{\substack{j=1 \\ j\neq i}}^{2M+1} \binom{(j-i)\alpha + N -h -1}{N} \right)^{-1}
    \end{equation}
    for all $1\leq i\leq 2M+1$ and $0\leq h\leq N$.
    A proof can be found in \cite[Chapter~IV, Theorem~3.1.1 and Equation~3.1.9]{Jag2}\footnote{We follow the notations and the setting of the Padé approximation scheme in \cite[p.66]{ChCh85}. To obtain equations \eqref{eq:Padé} and \eqref{eq:pih} one replaces $m,\rho_\mu,\sigma$ in \cite{Jag2} by $2M+1,N,\sigma+1$, and multiplies by~$N!^{2M+1}$.}.

    Note that the assumption on $\alpha$ not being a rational number is necessary for the expressions $p_{i,h}$ for large $N$ to be defined. At the same time for rational $\alpha$ the function $x^\alpha$ is algebraic, and there exist $M$ and a type, such that the remainder of the Hermite-Padé approximation vanishes. 

    In particular we know that the first nonzero coefficient in equation \eqref{eq:Padé} is the coefficient of $z^{\sigma},$ which we will denote by $g_\sigma$. We obtain a second expression for it, by expanding the left hand side of \eqref{eq:Padé}:
    \begin{equation}\label{eq:coeffsigma}
        g_\sigma \coloneqq \sum\limits_{\mathclap{\substack{1\leq i\leq 2M+1 \\ 0\leq h\leq N}}}\ (-1)^{\sigma-h} \binom{(i-1)\alpha}{\sigma-h} p_{i,h} = \frac{N!^{2M+1}}{\sigma!}.
    \end{equation}
    In the following we will make use of the fact that the norm of a non-zero algebraic integer is at least $1$. More precisely, if $L$ is a number field of degree $d$ over $\Q$, then for any $\gamma\in L$ we have 
    \begin{equation}\label{eq:Liouville}
        |\den(\gamma)^d\Norm_{L/\Q}(\gamma)| \geq 1.
    \end{equation}
    
    We will apply this inequality to $g_\sigma\cdot\Omega$, where $\Omega= \Omega_{M, N}$ is given by 
    \begin{equation} \label{eq:Omega}
        \Omega \coloneqq  \left(\prod\limits_{k=1}^{M} \omega_{k}\right)\cdot\left(\prod\limits_{k=1}^{2M} \omega_{k}\right)\in L, \quad \text{with} \quad \omega_k\coloneqq \binom{k\alpha+N}{N}(k\alpha)\binom{-k\alpha+N}{N} ,
    \end{equation}
    to obtain the announced contradiction.

    \begin{rem}
        Our definition of $\Omega$ differs slightly from the quantity called $\Omega$ in the Chudnovsky brothers' article~\cite{ChCh85} because we believe that their line of argument, that we essentially followed and reproduced, does not hold with their choice.
        More precisely, the factors $(\pm k\alpha+N)$ do not appear in their definition of $\Omega$, hence for any $i\in\{1,\dots, 2M+1\}$, the product $\Omega p_{i,N}$ simplifies to a polynomial in $\alpha$ with rational coefficients, times some factors of the form $(k\alpha+N)^{-1}$, for $k\in\N_{>0}$.
        The reasoning on the primes appearing in the denominator of this product does not hold anymore in this case (see the proof of Proposition \ref{prop:DNMbound}, Equations \eqref{eq:qomegaminus} and \eqref{eq:qomegaplus}).
    \end{rem}

    \subsubsection{Bounding the Denominator}
    
    In this section we will prove the following bound on the denominator of $\Omega g_\sigma$.
    
    \begin{prop}\label{prop:DNMbound}
        The denominator of $\Omega g_\sigma$ satisfies
        \begin{equation*}
            \den(\Omega g_\sigma) \leq C_0^{(2M-1)N} r_n^{6MN + 5M +N} \delta^{6 MN + N + 2M},
        \end{equation*}
        where $C_0$ satisfies $\lcm(1,\dotsc,N)\leq C_0^N$ and $\delta = \prod\limits_{p|\Delta_R} p^{1/(p-1)}$.
    \end{prop}

    \begin{rem}
        The existence of such a real number $C_0$ can be deduced from bounds on the Chebyshev function $\psi(x) = \log(\lcm(1,\dots,\lfloor x\rfloor))$.
        A possible $C_0$ is $\sqrt[113]{C} < 2.826$, where $C$ is an explicit 51-digit number, see \cite[Thm.~12]{RS62} and \href{https://oeis.org/A206431}{OEIS A206431}.
    \end{rem}

    The following lemma is the key ingredient in the proof of Proposition~\ref{prop:DNMbound}, and as such, also the key ingredient to reducing Kronecker's Theorem to a statement about a finite number of primes. We denote by $v_p(\cdot)$ the $p$-adic valuation.

    \begin{lem}\label{lem:binomred}
        Let $k,s\in\Z$, $t\in\N$, and let $p$ be a prime number. Let $\alpha$ be given as the zero of a square-free polynomial $R(w)\in \Z[w]$ with leading coefficient $r_n$ and $\Delta_R\coloneqq \res_w(R(w), R'(w))$.
        \begin{enumerate}[\quad (1)] \itemsep0em
            \item If $p\doesnotdivide \Delta_R$ and $R(w)\bmod p$ splits completely in $\F_p[w]$, then $p\doesnotdivide\den\left(\binom{k\alpha+s}{t}\right)$.
            \item If $p\doesnotdivide r_n$ and $p\vert\den\left(\binom{k\alpha+s}{t}\right)$, then $p\leq t$ and $v_p\left(\binom{k\alpha+s}{t}\right)\geq -\left \lfloor \frac{t}{p-1}\right \rfloor$.
            \item If $v_p(r_n)=v$, then $v_p\left(\binom{k\alpha+s}{t}\right)\geq -\left(v t + \left \lfloor \frac{t}{p-1}\right \rfloor\right)$.
            \item If $R(w)\bmod p$ splits completely in $\F_p[w]$ for all primes $p\leq t$ not dividing $r_n$, then $\den\left(\binom{k\alpha+s}{t}\right)$ divides $r_n^t \prod_{p\mid \Delta_R}p^{\left\lfloor \frac t {p-1}\right\rfloor}.$
        \end{enumerate}
    \end{lem}

    \begin{proof}
        If $R(w)\bmod p$ splits in $\F_p[w]$, and $p\doesnotdivide \Delta_R$ then all of its zeroes are distinct elements of $\F_p$. Then, by Hensel's Lemma \cite[(4.6) p.129]{Neu99} $R(w)$ splits over $\Z_p$ into linear factors, i.e. $\alpha \in \Z_p$. 
        It follows that $\binom{k\alpha +s}{t}\in\Q_p$ and we want to show that $v_p\left( \binom{k\alpha+s}{t} \right)\geq 0$.
        For all $j\in\N$ there exists $a_j\in\Z$ such that $\alpha=a_j+c_jp^j$ where $c_j\in\Z_p$.
        It follows that $\binom{k\alpha+s}{t}=\binom{ka_j+s}{t}+b_j\frac{p^j}{t!}$ for some $b_j\in\Z_p$. 
        Now $v_p\left( \binom{k\alpha+s}{t} \right)\geq \min\left( v_p\left(\binom{ka_j+s}{t}\right),v_p\left(b_j\frac{p^j}{t!}\right)\right)$.
        Since $a_j\in\Z$, $\binom{ka_j+s}{t}\in\Z$ and its $p$-adic valuation is non-negative for all $j\in\N$.
        By choosing $j\geq v_p(t!)$ we thus ensure that $v_p(p^j)\geq v_p(t!)$ and $v_p\left(\binom{k\alpha+s}{t}\right)\geq 0$.

        The second statement follows from the fact that $\binom{k\alpha+s}{t} = (k\alpha+s)_t/t!$.
        Thus the denominator of $\binom{k\alpha+s}{t}$ is composed of a product of primes dividing either $t!$ or the denominator of $\alpha$.
        Since $\den(\alpha)\vert r_n$, we notice that $p\doesnotdivide\den(\alpha)$ here.
        Primes dividing $t!$ are all primes smaller than $t$.
        Moreover it is a standard fact that $v_p(t!)\leq \frac {t} {p-1}$, hence the result.

        For the third statement note that $p^v (k\alpha+s-\ell)$ has positive $p$-adic valuation for all integers $\ell$. Thus $v_p\left(p^{v t}(k\alpha+s)_t\right)\geq 0$ and we conclude as for the second statement.

        The fourth assertion follows in a straightforward way from the previous three ones.
    \end{proof}

    \begin{ex}
        We illustrate that the condition $p\doesnotdivide \Delta_R$ in (1) that ensures the existence of a Hensel lift is necessary. Let $R(w) \coloneqq w^2- 2\cdot w + 2$ with leading coefficient $r_2 = 1$, and $\Delta_R = -4$. Let $\alpha = i+1$ be its root. Pick $p=2\mid \Delta_R$. Then $\binom{i+1}{2} = \frac{i-1}{2}$ is not $2$-integral.
    \end{ex}

   Let us recall some binomial identities which we will use in the following and can be proven by just rearranging factors. For any $x\in\CC$, and $\ell, m\in\N$, with $\ell \geq m$ we have
        \begin{equation}\label{eq:binomidentity}
            \binom{x}{\ell}\binom{\ell}{m} = \binom{x}{m}\binom{x-m}{\ell-m} = \binom{x}{\ell-m}\binom{x-\ell+m}{m}.
        \end{equation}

    \begin{proof}[Proof of Proposition \ref{prop:DNMbound}]
        Let $M,N\in\N_{>0}$.
        For $1\leq i\leq 2M+1$, and $0\leq h\leq N$ consider the coefficients $p_{i,h}$ given in \eqref{eq:pih}. In accordance with the middle expression of~\eqref{eq:coeffsigma}, we are aiming to bound $\lcm_{i,h}\left(\den\left(\Omega p_{i,h}\binom{(i-1)\alpha}{\sigma-h}\right)\right)$.
    
        Let us start by rewriting the factors appearing in $p_{i,h}$. By applying the equality between left and right-hand sides of~\eqref{eq:binomidentity} with $x\coloneqq (j-i)\alpha + N -h -1$, $\ell\coloneqq N$ and $m\coloneqq h$, we get
        \[\binom{N}{h} \binom{(j-i)\alpha + N -h -1}{N} = (-1)^h \binom{(j-i)\alpha+N-h-1}{N-h} \binom{(i-j)\alpha+h}{h} \eqqcolon q_{j-i, h}.  \] 
        Consequently we may write
        \[p_{i,h}=\binom{N}{h}^{2M+1}\cdot  \prod_{k=1}^{i-1}q_{-k, h}^{-1} \prod_{k=1}^{2M+1-i}q_{k, h}^{-1}.\]
        Let us note that either $i\leq M$ or $2M+1-i\leq M$ allowing us to match factors $\omega_k$ of $\Omega$ with the factors of $p_{i,h}$. We will thus treat the case $1\leq i \leq M$, with the case $M+1\leq i\leq 2M+1$ being completely analogous. We note that $\binom{-k\alpha+N-h}{N-h}(k\alpha) = - (-k\alpha+N-h)\binom{-k\alpha + N - h -1}{N - h}$ and, using the first equality of \eqref{eq:binomidentity} twice, we obtain the following equation
        \begin{equation} \label{eq:qomegaminus}
            (-1)^hq_{-k, h}^{-1} \omega_k = \frac{\binom{k\alpha+N}{N}}{\binom{k\alpha+h}{h}}\cdot \frac{\binom{-k\alpha+N}{N}(k\alpha)}{\binom{-k\alpha+N-h-1}{N-h}} = \frac{\binom{k\alpha+N}{N-h}}{\binom N h}\cdot \frac{-\binom{-k\alpha+N}{h}(-k\alpha+N-h)}{\binom N h},
        \end{equation}
        and analogously we have
        \begin{equation}\label{eq:qomegaplus}
            (-1)^hq_{k, h}^{-1} \omega_k = \frac{\binom{-k\alpha+N}{N-h}}{\binom N h}\cdot \frac{\binom{k\alpha+N}{h}(k\alpha+N-h)}{\binom N h}.
        \end{equation}

        In total, using \eqref{eq:qomegaminus} for $k<i\leq M$ and \eqref{eq:qomegaplus} for $k<2M+1-i\leq 2M+1$, we obtain an expression of $\Omega p_{i,h}$ of the form $\binom N h ^{-2M+1} b,$ where $b$ is a product of binomial coefficients of the form $\binom{\pm k\alpha +j}{h}$, $\binom{\pm k\alpha +j}{N-h}$ and linear factors of the form $(\pm k\alpha + j)$, consisting of the factors appearing in the numerators of the right-hand side of \eqref{eq:qomegaminus} and \eqref{eq:qomegaplus}, and the $M$ ``unused'' factors $\omega_k$ for $i\leq k\leq M$ and $2M+2-i\leq k \leq 2M$. The degree of $b$ in $\alpha$ is given by $2M (N+1) + (2N+1)M=4MN+3M$. Indeed, in the $2M$ factors coming from \eqref{eq:qomegaminus} and \eqref{eq:qomegaplus} the degree is $N+1$ each and the remaining $M$ factors $\omega_k$ are of degree $2N+1$ each. Moreover, using that $\lfloor (N-h)/s\rfloor + \lfloor h/s \rfloor \leq \lfloor N/s\rfloor$ for any $s\in\N_{>0}$, we obtain from Lemma~\ref{lem:binomred}~(4) that $\den(b)$ divides \[r_n^{4MN+3M}\prod_{p\vert \Delta_R}\left(p^{\left\lfloor \frac N {p-1}\right\rfloor}\right)^{4M} \left(\lcm_{h\in \{0,\ldots, N\}}\binom N h\right)^{2M-1}.\]
        Moreover, $\den\left(\binom{(i-1) \alpha}{\sigma -h}\right)$ divides the expression $r_n^{\sigma-h}\prod_{p\vert \Delta_R}\left(p^{\left\lfloor \frac {\sigma - h} {p-1}\right\rfloor}\right),$ by Lemma~\ref{lem:binomred}~(4), which in turn divides $r_n^{\sigma}\prod_{p\vert \Delta_R}\left(p^{\left\lfloor \frac {\sigma} {p-1}\right\rfloor}\right).$

        Putting things together, for any pair $(i,h)$ we have that $\den(\Omega p_{i,h} \binom{(i-1) \alpha}{\sigma -h})$ divides \[r_n^{6MN + 5M +N} \prod_{p\vert \Delta_R}\left(p^{\left\lfloor \frac {6 MN + N + 2M} {p-1}\right\rfloor}\right)\left(\lcm_{h\in \{0,\ldots, N\}}\binom N h\right)^{2M-1}.\] 
        We estimate $\lcm_{h\in\{0,\ldots, N\}} \binom{N}{h}\leq \lcm\{1,\ldots, N\}\leq C_0^N$, where the first inequality can, for instance, be deduced from the more general identity $(N+1)\lcm_{h\in\{0,\ldots, N\}} \binom{N}{h}= \lcm\{1,\ldots, N+1\}$, see \cite{Far09}. We obtain
        \[\den(\Omega g)\leq r_n^{6MN + 5M +N} \delta^{6 MN + N + 2M}C_0^{(2M-1)N},\] as announced.
        \end{proof}

    \subsubsection{Bounding the Norm}\label{sub:coeffBound}

    In this section we want to obtain a bound on $|\Norm(\Omega g_\sigma)|,$ where $g_\sigma=\frac{N!^{2M+1}}{\sigma!}.$
    Our computations will involve the sum of all integers from $1$ to $2M$, with those up to $M$ counted twice, as well as the sum of squares of these numbers. Thus we define the corresponding polynomials $u_1(M)$ and $u_2(M)$ to simplify our expressions
    \begin{align*}
        u_1(M) &\coloneqq \sum\limits_{k=1}^{2M} k + \sum\limits_{k=1}^{M} k = \frac{5}{2}M^2+\frac{3}{2}M \\
        u_2(M) &\coloneqq \sum\limits_{k=1}^{2M} k^2 + \sum\limits_{k=1}^{M} k^2 =3M^3+\cfrac{5}{2}M^2+\cfrac{1}{2}M.
    \end{align*}

    Also, we will make use of the following convenient fact about computing norms in a Galois extension.

    \begin{rem}\label{rem:Norm}
        Let us remind that if $L/\Q(\alpha)/\Q$ is a Galois extension, and $\phi(\alpha)\in\Q(\alpha)$ is a rational expression in $\alpha$, then $\Norm_{L/\Q}(\phi(\alpha))\in\Q$ can be computed as follows, up to its sign.
        We know that $\Norm_{L/\Q}(\phi(\alpha)) = \prod_{\sigma\in\mathrm{Gal}(L/\Q)} \phi(\sigma(\alpha))$.
        Let $\pi_\alpha(x)\in\Q[x]$, be the monic minimal polynomial of $\alpha$ of degree $e$ dividing  $d\coloneqq[L:\Q]$, then $\Norm_{L/\Q}(\phi(\alpha))=\prod_{\tilde \alpha}\phi(\tilde\alpha)^{d/e}$, where the product ranges over all roots $\tilde\alpha$ of $\pi_\alpha(x)$.
        In particular if $r\in\Q$, then $\Norm_{L/\Q}(r) = r^d$.
    \end{rem}
    
    \begin{prop}\label{prop:ANMfinalbound}
        Let $B$ be an upper bound on the modulus of roots of $R(w)$, let $M\in\N$ and $N\geq 4MB$. Then
        \begin{equation*}
            |\Norm(\Omega g_\sigma)|^{1/d} < \frac{(2e^2)^M \cdot \exp\left(\pi u_1(M) B+4\log\left(\frac 4 3\right) \frac{u_2(M)}{N} B^2 +\frac{M+1}{6N}\right)}{(2M+1)^{(2M+1)N} \cdot N^{M} \cdot (2M+1)^{2M+1/2} \cdot \pi^{2M}}.
        \end{equation*}
    \end{prop}

    To obtain bounds on the factorials, we borrow the following variant of Stirling's formula from \cite{Rob55}: For $k\in\N_{>0}$ we have
        \begin{equation}
            \left(\frac{k}{e}\right)^k \cdot \sqrt{2\pi k} \cdot \exp\left(\frac{1}{12k+1}\right) < k! < \left(\frac{k}{e}\right)^k \cdot \sqrt{2\pi k} \cdot \exp\left(\frac{1}{12k}\right).
        \end{equation}

    We deduce the following bound on $g_\sigma$ by routine estimates; we mainly use that $\sigma>(2M+1)N$ in the denominators but not in the exponents. Thus we prove half of Proposition~\ref{prop:ANMfinalbound}.

    \begin{lem}\label{lem:ANMfactorialbound}
        Let $M,N\in\N$ and write $\sigma=(2M+1)N+2M$. Then
        \begin{equation*}
            \frac{N!^{2M+1}}{\sigma!} < (2M+1)^{-(2M+1)N} N^{-M} (2M+1)^{-2M-1/2} \left(2\pi e^2\right)^M \exp\left(\frac{M+1}{6N}\right).
        \end{equation*}       
    \end{lem}

    The second part of the proof of Proposition~\ref{prop:ANMfinalbound} is using Euler's product formula for the sine function, stating that for $z\in \CC$ we have
    \begin{equation} \label{eq:Euler} 
            \sin(z) = z \prod\limits_{j\geq 0} \left(1 - \frac{z^2}{j^2\pi^2}\right).
    \end{equation}    

    \begin{lem}\label{lem:Omegabound}
            Assuming that $N\geq 4MB$ we have
            \[|\Norm(\Omega)|^{1/d} \leq \frac 1 {\pi^{3M}} \cdot \exp \left( \pi B u_1(M) + 4\log(4/3)\frac{B^2}{N} u_2(M)\right).\]
    \end{lem}

    \begin{proof}
     For a fixed integer $k\in\{1,\ldots, 2M\}$ we can rewrite the corresponding factor appearing in $\Omega$ using \eqref{eq:Euler} as
        \begin{equation*}
            \binom{k\alpha+N}{N}(k\alpha)\binom{-k\alpha+N}{N} = (k\alpha) \prod\limits_{j=1}^N \left(1-\frac{k^2\alpha^2}{j^2}\right)=\frac{\sin(k\pi \alpha)} \pi \ \prod _{\mathclap{j\geq N+1}} \ \left(1-\frac{k^2\alpha^2}{j^2}\right)^{-1}.
        \end{equation*}
        
        We now wish to bound the sine by the exponential function, and estimate the remaining product. For any $z\in\CC$, we have $|\sin(z)| \leq \exp(|z|)$ hence $|\sin(k\pi\alpha)| \leq \exp(k\pi B)$,
        as $B$ is a bound on the modulus of $\alpha$.

        For the remaining product, let us now write $x_j\coloneqq\frac{k^2\alpha^2}{j^2}$. By our assumption on $N$, we have that $|x_j|\leq \frac{1}{4}$. We then make use of the fact that the function $f(x)\coloneqq(1-x)^{1/x}$ is strictly decreasing on the interval $(0,1)$ and deduce that $0<|x_j|\leq 1/4$ implies $f(1/4) \leq f(|x_j|)$, which is, after applying the logarithm on both sides, equivalent to $-\log(1-|x_j|) \leq 4 \log (4/3)\cdot |x_j|$. We obtain the bound
        \begin{align*}
            \left|\prod_{j\geq N+1} (1-x_j)^{-1}\right| & \leq  \prod_{j\geq N+1} (1-|x_j|)^{-1} = \exp\left(\sum_{j\geq N+1} -\log(1-|x_j|)\right)\\
            & \leq \exp\left(4\log\left(\frac{4}{3}\right)\cdot \sum_{\mathclap{j\geq N+1}} |x_j| \right) \leq \exp\left(4\log\left(\frac{4}{3}\right)\cdot \frac{k^2B^2}{N}\right),
        \end{align*}
        where the last inequality is justified by  $|x_j|\leq (kB)^2/j^2$, and $\sum_{j\geq N+1}\frac{1}{j^2}\leq \frac 1 N$.

        We take the product of the bounds obtained for each individual $k$ for $k\in\{1,\dots,2M\}$ and, again, for $k\in\{1,\dots,M\}$ to obtain \[|\Omega| \leq \frac 1 {\pi^{3M}} \cdot \exp \left( \pi B u_1(M) + 4\log(4/3)\frac{B^2}{N} u_2(M)\right).\] All the estimates above are equally valid when replacing $\alpha$ by any root of $R(w)$, so in particular for the conjugates of $\alpha$. Thus the estimates also hold for any conjugate of $\Omega$, and, by Remark~\ref{rem:Norm}, we conclude.
    \end{proof}

    Putting together Lemmata \ref{lem:ANMfactorialbound} and \ref{lem:Omegabound} finishes the proof of Proposition~\ref{prop:ANMfinalbound}.

    \subsection{A Bound on the Number of Primes} \label{sub:conclusion}

    Finally, combining the results of Propositions \ref{prop:DNMbound} and \ref{prop:ANMfinalbound} we get an upper bound on the integer $|\den(\Omega g_\sigma)^d\Norm_{L/\Q}(\Omega g_\sigma)|$.  
    To prepare for further estimates, we split it into three parts and set 
    \begin{equation*}\arraycolsep=1.4pt\def\arraystretch{2.2}
    \begin{array}{cl}
        X(M) &\coloneqq \cfrac{C_0^{2M-1}r_n^{6M+1} \delta^{6M+1}}{(2M+1)^{2M+1}} = \cfrac{1}{C_0^2r_n^2\delta^2}\cdot\left(\cfrac{C_0r_n^3\delta^3}{2M+1}\right)^{2M+1}\\
        Y(M, N) &\coloneqq \exp\left(\pi B u_1(M)+ 4\log\left(\cfrac{4}{3}\right)B^2 \cfrac{u_2(M)}{N}\right) \cdot (2M+1)^{-2M-1/2} \cdot \left( \cfrac{2e^2 r_n^5 \delta^2}{\pi^2} \right)^M\\
        Z(M, N) &\coloneqq \left( \cfrac{1}{N} \right)^M \cdot \exp\left(\cfrac{M+1}{6N}\right).
    \end{array}
    \end{equation*}
    The factor $X(M)$ consists of factors from the upper bound of Propositions~\ref{prop:DNMbound} and~\ref{prop:ANMfinalbound} that are raised to the power $N$. In this sense, it is the main contribution of the upper bound as $N$ grows. The factor $Y(M,N)$ consists of non-negligible factors with respect to $M$, its behavior as $M$ grows will determine how large we need $N$ to be with respect to $M$ for our bound to become smaller than $1$. Finally, in $Z(M,N)$ we group factors that are negligible for the growth behavior of the bound.

    Of course, these quantities also depend on $r_n$, and $\Delta_R$, which are suppressed in the notation. Moreover, once we fix $M$ to be a function of $r_n, \Delta_R$, the notation becomes particularly misleading. However, in the following calculations it will be convenient to view these quantities as functions of $M$ and $N$.
    
    \begin{cor}\label{cor:AnD}
        Let $M\in\N$ and $N\geq 4MB.$ Then
        \begin{equation}\label{eq:AnD}
            0<|\den(\Omega g_\sigma)\Norm_{L/\Q}(\Omega g_\sigma)^\frac 1 d| < X(M)^N\cdot Y(M,N) \cdot Z(M, N).
        \end{equation}
    \end{cor}

    From the asymptotics of these bounds, one can easily deduce the existence of parameters $M$ and $N$, such that $|\den(\Omega g_\sigma)\Norm_{L/\Q}(\Omega g_\sigma)^\frac 1 d|<1$. For large enough $M$ compared to $r_n, \Delta_R$, the expression $X(M)$ is strictly smaller than $1$. Moreover, if we choose $N$ large enough with respect to $M$, we have $X(M)^N<Y(M, N)^{-1}$, as the right-hand side is increasing with $N$. Finally, for $N>M>3$, clearly $Z(M, N)<1$. Thus, for this choice of $M, N$ we find the desired contradiction to the assumption that $\alpha$ was irrational, and so we have proved Kronecker's Theorem. This argumentation was carried out by the Chudnovsky brothers in \cite{ChCh85}. The rest of this subsection is devoted to showing that our explicit choices of $M$ and $N$ suffice for the contradiction to occur.

    \subsubsection{Choice of the Number of Functions to Approximate} \label{subsub:choiceM}

    Let us first justify our choice of $M$. 

    \begin{lem} \label{lem:Mpower2}
        For any $M>\frac{1}{2}(C_0r_n^3\delta^3-1)$ we have $X(M)<1$. Moreover, for $M>C_0r_n^3\delta^3$ we have $X(M)<2^{-(2M+1)}(C_0r_n \delta)^{-2}$.
    \end{lem}

    \begin{proof}
        The proof consists of a straightforward computation.
    \end{proof}

    The choices of $M$ and $N$ are dependent on each other in the following way: the function $X(M)$ is decreasing with $M$, while simultaneously $Y(M,N)$ increases. Thus, the first value of $N$, such that $X(M)^NY(M,N)Z(M, N)<1$ varies in a non-obvious way with $M$. To obtain tight bounds one should minimize $\sigma=(2M+1)N+2M$ under the constraint $X(M)^NY(M, N)Z(M, N)<1$. We carry out no such computations, but are content with $M= \lceil C_0r_n^3\delta^3\rceil$. Na\"{\i}ve estimates and computations suggest that picking $M\approx\lceil 0.806 \cdot C_0r_n^3\delta^3\rceil$ yields slightly more optimal values for $\sigma$.

    \subsubsection{Choice of the Type in the Approximation}

    The goal of this section is to prove that for the explicit values of $M$ and $N$ given in Theorem~\ref{thm:intro1} the right-hand side of \eqref{eq:AnD} is bounded from above by $1$. We assume $M=\lceil C_0r_n^3\delta^3\rceil$ to be fixed. In Proposition~\ref{prop:ANMfinalbound}, we assume $N$ to depend at least linearly on $B$ and $M$. Let us set $N\coloneqq ABM$ for some constant $A\geq 4$ to be determined later. With these choices of $M$ and $N$ the right-hand side of~\eqref{eq:AnD} is a function in $r_n$ and $\Delta_R$. However, to simplify computations we will treat it as a function in $M=\lceil C_0r_n^3\delta(\Delta_R)^3\rceil$. 
    
    \begin{prop} \label{prop:smaller1}
        For $M=\lceil C_0r_n^3\delta^3\rceil$ and $N=ABM$ with $A\geq 4$, we have 
        \begin{equation}
            X(M)^NY(M, N)Z(M, N)<\exp\left(T(M) \right),
        \end{equation}
        where $T(M)\coloneqq c_2 M^2 + \tilde c_1 M \log(M) + c_1M + \tilde c_0 \log(M) + c_0$ with 
        \begin{gather*}
            c_2  \coloneqq \left(-2A\log(2)+ \frac{5\pi}{2} + \frac{12\log(4/3)}{A}\right) B,\qquad \tilde c_1 \coloneqq -1-\frac{2}{3}AB \\
            c_1  \coloneqq\left(- A\left(\frac 2 3 \log\left(\frac{2C_0^{2}}{3}\right) +\log(2)\right)+ \frac{3\pi} 2 + \frac{10\log\left (\frac 4 3\right)}{A}\right) B,\\ \tilde c_0 \coloneqq -\frac 1 2 , \qquad c_0 \coloneqq \frac 1 {18} + \frac{2\log\left(\frac 4 3 \right)B} A  .
        \end{gather*}

        For $A> 6.076$ and all $B\geq 1$ and $r_n\geq1$ the expression $T(M)$ is negative and thus \[X(M)^NY(M, N)Z(M, N)<1.\]
    \end{prop}

    \begin{proof}
        We estimate crudely $\frac{2e^2r_n^5 \delta^2}{N \pi^2}<M$. Moreover, by Lemma~\ref{lem:Mpower2}, $X(M)<(C_0r_n \delta)^{-2}\cdot2^{-(2M+1)}$.
        We have $-\log(C_0^2r_n^2 \delta^2)=-\frac{4}{3}\log(C_0)-\frac{2}{3}\log(C_0r_n^3\delta^3)$, and $-\log(C_0r_n^3\delta^3)\leq -\log(M-1) \leq -\log(M) -\log(\frac 2 3)$ by factoring $M$ in the $\log$ and using that $M\geq 3$, hence $-\log(C_0^2r_n^2 \delta^2)\leq -\frac{2}{3}\log(\frac{2C_0^2}{3})-\frac{2}{3}\log(M)$. We further use $\log(2M+1)>\log(M)$ and bound $\frac{M+1}{6N}<\frac 1 {18}$. The rest of the proof of the first part is obtained from the definitions of $X(M)$ and $Y(M, N)$ by straightforward computations.

        For the second statement we first aim to show that $T(M)< c_2M^2$, under our assumption $A\geq 4$.
        Solving a quadratic equation in $A$, we convince ourselves that $c_1$ is always negative when $A\geq 4$. Moreover, $\tilde c_1$ and $\tilde c_0$ are negative, but $c_0$ is positive. As $M\geq 3$, we note that $-\frac{1}{2}\log(M)+\frac{1}{18}<0$, and $3\log(3)\leq M\log(M)$ hence $-\frac{2}{3}AM\log(M)+\frac{2\log\left(\frac 4 3 \right)} A<0$ for $A\geq 1$. So $\tilde c_1 M \log(M) + \tilde c_0 \log(M) + c_0<0$ as well.

        Finally, we aim to find $A\geq 4$, such that $c_2<0$. Finding this value of $A$ amounts to solving a quadratic equation in $A$, yielding $c_2<0$ for $A>6.076$.
    \end{proof}

    \section{Algorithm and Complexity} \label{sec:algo}

    The bounds we present in Theorem~\ref{thm:intro2} allow us, as was our purpose, to solve algorithmically the problem of deciding the algebraicity of solutions of equations \eqref{eq:deq1}.
    We present here an algorithm, Algorithm~\ref{algo:Honda}, solving this problem and we estimate its complexity. Each of its steps is studied in what follows, and the full algorithm's complexity estimate is analyzed in Section~\ref{ssec:fullalgo}. It turns out to be inefficient, compared to other, more straight-forward, approaches for certifying algebraicity, cf. Corollary~\ref{cor:ratroots}. All complexities are stated in number of bit operations.

    Throughout this section, we consider our input rational function $u(x)$ to be of \emph{normal form} $\frac{a(x)}{b(x)}$ where $a(x)$ and $b(x)$ have integer coefficients, are primitive, coprime, and $\deg a(x)<\deg b(x)$.
    The primitivity condition can be made without loss of generality, as already discussed in the proof of Theorem~\ref{thm:intro2} in Section~\ref{sec:Reduction}. This comes from the fact that if $y(x)$ satisfies $y'(x)=u(x) y(x)$, then for any $c\in\Q$, the function $\tilde y(x)\coloneqq y(x)^c$ satisfies $\tilde y'(x)=c\cdot u(x) \tilde y(x)$, and $y(x)$ is algebraic if and only if $\tilde y(x)$ is algebraic.
    
    Our measure of complexity of a rational number will be an estimate of ``its size on a computer'', that we call its \emph{height}. More precisely, if $q=a/b\in\Q$ with $a,b\in\Z\setminus\{0\}$ coprime, the height of $q$ is $h(q)\coloneqq\log_2(|ab|)$, and we set the height of $0$ to be $1$.
    The height $h(f(x))$ of $f(x)\in\Q[x]$ is the maximum of the heights of its coefficients, and its \emph{multiplicative height} is $H(f(x))\coloneqq 2^{h(f(x))}$.
    Remark that if $f(x)$ has integer coefficients, then $H(f(x))$ is the maximum of the modulus of its coefficients.
    The height of a rational function $u(x)=a(x)/b(x)$ with $a(x),b(x)\in\Z[x]$ primitive and coprime, is the maximum of $h(a(x))$ and $h(b(x))$.

    The computation of the normal form, as described above, of $u(x)\in\Q(x)$ of height $h$ and degree $n-1$ in the numerator and $n$ in the denominator requires computing the greatest common divisor of $n+1$ numbers and the greatest common divisor of two polynomials.
    The computation of the greatest common divisor of two polynomials in $\Q[x]$ of degree at most $n$ and height $h$ can be made by performing $\tilde O(nh)$ bit operations \cite[Corollary 11.9, p.325]{Gz13}.
    With an adaptation of the same algorithm we can compute the gcd of two integers of height at most $h$ in $\tilde O(h)$ bit operations.
    Hence applying it $n$ times allows to compute the greatest common divisor of $n+1$ integers of height at most $h$ in $\tilde O(nh)$ bit operations, thus computing the normal form of $u(x)$ can be made by performing $\tilde O(nh)$ bit operations.

    Additionally, we will assume in the following that the Rothstein-Trager resultant $R(w)= \res_x(b(x),a(x)-w\cdot b'(x))$ is square-free to lighten the exposition. If it is not, we replace it with its square-free part, which can be readily computed as $R(w)/\gcd(R(w),R'(w))$ in negligible bit operation cost. The polynomial $R(w)$ and its square-free part have the same roots, hence the same bound $B$ on their modulus, and the same prime factors of the leading coefficient. However, the height of $R(w)$ might be altered when taking the square-free part. We prove an estimate in Proposition~\ref{prop:factheight} that ensures the order of the height remains the same.

    Last, to avoid costly integer factoring of the potentially large resultant $\Delta_R$ to compute $\delta$, we prove an estimate in Proposition~\ref{prop:Delta2}. 

     \begin{algorithm}
        \caption{Deciding algebraicity with $p$-curvatures.}\label{algo:Honda}
        $\begin{array}{ll}
            \textbf{Input: }& \text{Polynomials } a(x),b(x)\in\Q[x] \text{ of degree at most $n$ and height at most $H$. } \\ 
            \textbf{Output: }& \text{The nature \texttt{Algebraic} or \texttt{Transcendental} of solutions of } y'(x)=\frac{a(x)}{b(x)}y(x).
        \end{array}$
        
        \begin{algorithmic}[1]
        \If{$\deg a(x)\geq \deg b(x)$ \textbf{or} $b(x)$ is not square-free} \Return \texttt{Transcendental};
        \EndIf
        \State Compute the normal form of $a(x)/b(x)$;
        \State $R(w)\leftarrow \texttt{SquareFreePart}(\res_x(b(x),a(x)-w\cdot b'(x)))$;
        \State $\Delta_b\leftarrow\texttt{LeadingCoefficient}(R(w))$;
        \State $\delta\leftarrow$ right-hand side of \eqref{eq:deltabound};
        \State $p\leftarrow 2$;
        \While{$p\leq \Delta_b$}
            \If{$\Delta_b \bmod p \neq 0$ \textbf{and} $\texttt{pCurvature}(a(x),b(x),p)\neq 0$} \Return \texttt{Transcendental};
            % $\delta \leftarrow \delta \cdot p^{1/(p-1)}$
            % \ElsIf{$\texttt{pCurvature}(a(x),b(x),p)\neq 0$} \Return \texttt{Transcendental};
            \EndIf
            \State $p\leftarrow \texttt{nextprime}(p)$
        \EndWhile
        % \While{$p\leq \Delta_R$} 
        %     \If{$\Delta \bmod p = 0$} $\delta \leftarrow \delta \cdot p^{1/(p-1)}$
        %     \EndIf
        %     \If{$\texttt{pCurvature}(a(x),b(x),p)\neq 0$} \Return \texttt{Transcendental};
        %     \EndIf
        %     \State $p\leftarrow \texttt{nextprime}(p)$
        % \EndWhile
        \State Compute $B$, $M\leftarrow \lceil 2.826 \Delta_b^3 \delta^3 \rceil$, $N\leftarrow \lceil6.076BM\rceil$;
        \State $\sigma\leftarrow (2M+1)N+2M$;
        \While{$p\leq\sigma$}
            \If{$\texttt{pCurvature}(a(x),b(x),p)\neq 0$} \Return \texttt{Transcendental};
            \EndIf
            \State $p\leftarrow \texttt{nextprime}(p)$;
        \EndWhile
        \State \Return \texttt{Algebraic};
        \end{algorithmic}
    \end{algorithm}

    \begin{rem} \label{rem:Qbalgo}
        As explained in Remark \ref{rem:Qbfact}, starting from $u(x)\in\Qb(x)$, we can check if $\tilde R(w)\coloneqq R(w)/\Delta_b$ is in $\Q[w]$ after step 4.
        If this is not the case we can immediately conclude that the output has to be \texttt{Transcendental} with Proposition \ref{prop:Qb}.
        % Otherwise, we take $\Delta_R$ to be the discriminant of the polynomial in $\Z[w]$ associated to $\tilde R(w)$.
        Otherwise, we can continue from step 5 after replacement of $R(w)$ by the polynomial in $\Z[w]$ associated to $\Tilde{R}(w)$ obtained by multiplication by the denominator of the coefficients.
    \end{rem}

    \subsection{Complexity of Rothstein-Trager Resultants} \label{ssec:RT}

    Let $u(x)=a(x)/b(x)\in\Q(x)$ be in normal form with $\deg a(x)<n\coloneqq\deg b(x)$, with the denominator $b(x)$ being square-free, height bounded by $h$ and multiplicative height at most $H=2^h$.

    If we write $b(x)=\sum_{i=0}^n b_ix^i$ and $a(x)=\sum_{i=0}^{n-1} a_ix^i$ then the Rothstein-Trager resultant $R(w)\coloneqq \res_x(b(x),a(x)-w\cdot b'(x))$ is the determinant of the $(2n-1)\times (2n-1)$ matrix
    \begin{equation}\label{eq:RTresultant}
        \begin{pmatrix}
            b_n & & &              & a_{n-1}-nwb_n & & & & \\
            b_{n-1} & b_n & &              & \vdots & a_{n-1}-nwb_n & & & \\
            \vdots & b_{n-1} & \ddots &               & \vdots & \vdots & \ddots & &  \\
            \vdots & \vdots & \ddots & b_n           & a_1-2wb_2  & \vdots & & \ddots & \\
            b_1 & \vdots &  & b_{n-1}             & a_0-wb_1 & a_1-2wb_2 & & & a_{n-1}-nwb_n \\
            b_0 & b_1 & & \vdots              & & a_0-wb_1 & \ddots & & \vdots \\
             & b_0 & \ddots & \vdots              & & & \ddots & \ddots & \vdots \\
             &  & \ddots & b_1              & & & & \ddots & a_1-2wb_2 \\
             &  &  & b_0                   & & & & & a_0-wb_1
        \end{pmatrix}.
    \end{equation}
    From the form of this \emph{Sylvester matrix}, we see that the degree of $R(w)$ is at most $n$, and its coefficient of degree $n$ is $\res(b(x),b'(x))$, up to a sign.
    The assumption that $b(x)$ is square-free ensures $\deg(R(w))=n$.
    Concerning the height of $R(w)$ we prove the following estimate.    
    \begin{prop}\label{prop:RTheight}
        For any $0\leq k\leq n$ the coefficient $r_k$ of $w^k$ in $R(w)=\res_x(b(x),a(x)-w\cdot b'(x))$ satisfies the inequality
        \begin{equation*}
            |r_k| \leq \binom{n}{k} 6^{-k/2} H^{2n-1} (n+1)^{(n+k-1)/2} n^{n/2} (2n+1)^{k/2}.
        \end{equation*}
    \end{prop}

    \begin{rem}
        The determinant is linear with respect to its columns, hence we can write the last columns of the matrix \eqref{eq:RTresultant} each as the sum of a part with the coefficients of $a(x)$, and a part with the coefficients of $b'(x)$, and factor $w$ in each column it appears. See Example~\ref{ex:linearity} below for an illustration.
        The problem of bounding the height becomes bounding a determinant of a particular shape with integer coefficients and estimating the number of terms for each power of $w$ that appears.
    \end{rem}

    \begin{ex}\label{ex:linearity}
        In the case $n=2$, $b(x)=b_0+b_1x+b_2x^2$ and $a(x)=a_0+a_1x$, the polynomial $R(w)=\res_x(b(x),a(x)-w\cdot b'(x))$ can be expanded as follows
        \begin{align*}
            R(w) & = 
            \begin{vmatrix}
                b_2 & a_1-2wb_2 & 0 \\
                b_1 & a_0-wb_1 & a_1-2wb_2 \\
                b_0 & 0 & a_0-wb_1
            \end{vmatrix} \\
            & = \begin{vmatrix}
                b_2 & a_1 & 0 \\
                b_1 & a_0 & a_1 \\
                b_0 & 0 & a_0
            \end{vmatrix} 
            -w\left(\begin{vmatrix}
                b_2 & 2b_2 & 0 \\
                b_1 & b_1 & a_1 \\
                b_0 & 0 & a_0
            \end{vmatrix}
            +\begin{vmatrix}
                b_2 & a_1 & 0 \\
                b_1 & a_0 & 2b_2 \\
                b_0 & 0 & b_1
            \end{vmatrix}\right)
            +w^2\begin{vmatrix}
                b_2 & 2b_2 & 0 \\
                b_1 & b_1 & 2b_2 \\
                b_0 & 0 & b_1
            \end{vmatrix} 
        \end{align*}
    \end{ex}

    The main ingredient for our estimates is the following lemma known as Hadamard's inequality \cite[Thm.~16.6, p.~477]{Gz13}.
    \begin{lem}\label{lem:Hadamard}
        Let $M$ be a $n\times n$ matrix over $\R$, with columns $C_1,\dots,C_n$, and coefficients bounded by $B>0$.
        Then
        \begin{equation*}
            |\det(M)| \leq ||C_1||\dots||C_n||
        \end{equation*}
        where $||C||$ denotes the $2$-norm of the vector $C\in\R^n$.
    \end{lem}

    \begin{proof}[Proof of Proposition \ref{prop:RTheight}.]
        Following the idea of using linearity presented in the Example~ \ref{ex:linearity}, we can write each coefficient $r_k$ as a sum of determinants.
        Those determinants have a left block of $n-1$ columns with the coefficients of $b(x)$, and a right block of $n$ columns, each being either coefficients of $a(x)$ or coefficients of $b'(x)$.
        Note that the power of $w$ in front of such a determinant is equal to the number of columns corresponding to coefficients of $b'(x)$.
        We will use Hadamard's inequality, Lemma~\ref{lem:Hadamard}, on each of these determinants, hence let us compute a sharper bound on the $2$-norms of each possible column.
        A column $C_1$ of coefficients of $b(x)$ has at most $n+1$ nonzero coefficients, each bounded by $H=\exp(h)$, hence $||C_1|| = \left( \sum\limits_{i=0}^n b_i^2 \right)^{1/2} \leq (n+1)^{1/2}H$.
        Similarly, a column $C_2$ of coefficients of $a(x)$ has at most $n$ nonzero coefficients, each bounded by $H$, hence $||C_2|| \leq n^{1/2}H$.
        Finally, if $C_3$ is a column of coefficients of $b'(x)$, then $||C_3|| \leq \left( \sum\limits_{i=1}^n (ib_i)^2 \right)^{1/2} \leq \left(\frac{n(n+1)(2n+1)}{6}\right)^{1/2}H$.

        Let us take any $0\leq k \leq n$, then the coefficient $r_k$ of $w^k$ in $R(w)$ is a sum of determinants, each of which consists of $n-1$ columns of coefficients of $b(x)$, of $n-k$ columns of coefficients of $a(x)$ and of $k$ columns of coefficients of $b'(x)$, and there are $\binom{n}{k}$ such determinants.
        Each of these has the same upper bound given by $||C_1||^{n-1}||C_2||^{n-k}||C_3||^k$.
        Combining this and the bounds above on those norms concludes the proof.
    \end{proof}

    Knowing its height, we can estimate the complexity of computing the Rothstein-Trager resultant.

    \begin{prop}\label{prop:RTcomplexity}
        The Rothstein-Trager resultant $R(w)=\res_x(b(x),a(x)-w\cdot b'(x))$ can be computed using $\tilde O(n^3h)$ bit operations.
    \end{prop}

    \begin{proof}
        The coefficient of $x^{n-1}$ in $a(x)-w\cdot b'(x)$ cannot vanish hence $\deg_x(a(x)-w\cdot b'(x))=n-1$.
        By Proposition \ref{prop:RTheight}, the height $H_R$ of the Rothstein-Trager polynomial satisfies $\log(H_R)=\tilde O(n h)$.
        Thus taking a prime $p>2H_R$, knowing the Rothstein-Trager resultant in $\F_p[w]$ allows us to uniquely reconstruct it in $\Q[w]$.
        Corollary 11.21 from \cite{Gz13} states we can compute this polynomial in $\F_p[w]$ in $\tilde O(n^2)$ operations in $\F_p$, each with a cost $\tilde O(\log(p))$, thus in total we can compute $R(w)\in\Q[w]$ by performing $\tilde O(n^3h)$ bit operations.
    \end{proof}

    Taking the square-free part of the Rothstein-Trager resultant can alter its height, let us bound the height of any of its factors.

    \begin{prop} \label{prop:factheight}
        Any factor of $R(w) = \res_x(b(x),a(x)-w\cdot b'(x))$ has height at most
        \begin{equation*}
            \frac{2^{2n+1}}{\pi n} \cdot H^{2n-1} \cdot (n+1)^{2n}.
        \end{equation*}
    \end{prop}

    \begin{proof}
        Let $H_R$ be the height of $R(w)\coloneqq r_nw^n+\dots+r_0$.
        Let $\Vert R(w)\Vert_2 = \left(\sum_{k=0}^n |r_k|^2\right)^{1/2}$.
        Assume $P(w)\in\Z[w]$ is a factor of $R(w)$, Theorem 2 in \cite{Mig74} ensures that for the height $H_P$ of $P(w)$ we have 
        \begin{equation*}
            H_P\leq \max\left( \binom{\ell}{k} ,\ 0\leq k\leq \ell\leq n \right) \Vert R(w)\Vert_2.
        \end{equation*}
        We have na\"{\i}vely $\Vert R(w) \Vert_2\leq (n+1)^{1/2}H_R$.
        The maximum for fixed $\ell$ is the central binomial coefficient $\binom{\ell}{\lfloor \ell/2\rfloor}$, which is an increasing sequence of $\ell$, thus $H_P \leq \binom{n}{\lfloor n/2\rfloor}(n+1)^{1/2}H_R$.
        Combining this with Proposition~\ref{prop:RTheight} we obtain
        $H_P\leq (n+1)^{2n} \binom{n}{\lfloor n/2 \rfloor}^2 H^{2n-1}$.
        The following upper bound on the central binomial coefficient yields the announced result
        \begin{equation*}
            \binom{n}{\lfloor n/2\rfloor} \leq \frac{2^{n+1/2}}{\sqrt{\pi n}}.
        \end{equation*}
        If $n$ is even this is a classical result, let us prove that it holds for any odd $n\in\N_{\geq 3}$.
        We have
        \begin{equation*}
            \binom{n}{\frac{n-1}{2}} = \binom{n-1}{\frac{n-1}{2}} + \binom{n-1}{\frac{n-3}{2}} 
            \leq \frac{2^{n-1/2}}{\sqrt{\pi (n-1)}}  \cdot \left(1 + \frac{n-1}{n+1} \right).
        \end{equation*}
        The result follows by dividing by the desired bound and using that $x \mapsto \frac{1}{2}\left( \frac{x}{x-1}\right)^{1/2} \left( 1 + \frac{x-1}{x+1} \right)$ attains its minimum $\frac{3\sqrt{3}}{4\sqrt{2}}<1$ on $[2, +\infty) $ at $x=3$.
    \end{proof}

    \begin{rem} \label{rem:QbRTcomp}
        Starting from $u(x)\coloneqq \frac{a(x)}{b(x)}\in\Qb(x)\setminus\Q(x)$, the computations are more costly.
        Indeed, in a number field $\Q(\alpha)=\Q[t]/(\pi_\alpha(t))$ with $\pi_\alpha(t)\in\Q[t]$ the minimal polynomial of $\alpha\in\Qb$, elements are polynomials of degree smaller than $\deg(\pi_\alpha(t))$, hence algebraic number multiplication has the complexity of modular polynomial multiplication.
    \end{rem}

    \subsection{Complexity Estimates for Finding Rational Roots of the Rothstein-Trager Resultant}
    \label{ssec:factoringcomp}
    In Section~\ref{ssec:factoring} we noted algebraicity of the solutions of~\eqref{eq:deq1} can be decided by Algorithm~\ref{algo:ratroots}. We briefly investigate the complexity of this algorithm. For efficiently finding all the rational roots of a polynomial we have the following complexity~\cite[Prop.~21.22]{BCGLLSS17}.

    \begin{thm}
        Let $R(w)\in\Z[w]$ be a square-free, primitive polynomial of degree $n$ and height $h_R$.
        The computation of all rational roots of $R(w)$ can be performed, in $\tilde O(n^2h_R)$ bit operations.
    \end{thm}

    \begin{rem}
        Von zur Gathen and Gerhard only give a probabilistic algorithm performing the task in the same expected complexity~\cite[Thm.~15.21]{Gz13}. The algorithm described in~\cite{BCGLLSS17} uses a deterministic (potentially costly) algorithm to find roots of the polynomial modulo a deterministically defined \emph{small} prime number, and then lifts them to characteristic zero. It improves a result by Loos~\cite{Loo83}, who proposed a deterministic algorithm for finding rational roots in $\tilde O(n^3h)$ bit operations.
    \end{rem}

    Using the corresponding algorithm the cost of computing the Rothstein-Trager resultant is comparable to finding rational roots, see Proposition~\ref{prop:RTcomplexity}. We obtain the following Corollary.
    
    \begin{cor} \label{cor:ratroots}
        Given $u(x)\coloneqq\frac{a(x)}{b(x)}\in\Q(x)$ of degree $n$ and height $h$, deciding if all solutions of $y'(x)=u(x) y(x)$ are algebraic can be done by performing $\tilde O(n^3h)$ bit operations.
    \end{cor}

    \subsection{Computation of the Bound on the Primes} \label{ssec:algosigma}

    Given $u(x)$, and having computed the Rothstein-Trager resultant $R(w)\in\Z[w]$, Theorem~\ref{thm:intro1} gives an explicit bound $\sigma\in\N$ on the number of $p$-curvatures to check, depending on $r_n$, the leading coefficient, on $\Delta_R\coloneq \res_w(R(w),R'(w))$, and on an upper bound $B$ on the modulus of all roots of $R(w)$.

    An arbitrarily precise estimate of $B$ is possible without factoring completely the polynomial $R(w)$, for example by following ideas from an unpublished report of Schönhage in which an algorithm to find the complex roots of a polynomial $R(w)\in\CC[w]$ with arbitrary precision is presented \cite[Thm.~15.1]{Sch82}.
    
    \begin{prop}[Schönhage] \label{prop:rootradius}
        Let $R(w)\in\Z[w]$ of degree $n$ and $\tau>0$. The computation of $B_0>0$ such that the maximum modulus $r$ of all roots of $R(w)$ satisfies $B_0e^{-\tau}< r < B_0e^\tau$ can be done in $O(n^2 (\log(\frac{1}{\tau})+\log(\log(n)))\log(\frac{4}{\tau}))$ binary operations.
    \end{prop}
    The bound $B$ we are looking for is $B_0e^\tau$.
    For estimating the complexity of Algorithm~\ref{algo:Honda}, we will see that simply taking $\tau = 1/2$ suits us, hence the complexity of computing the bound $B$ is $O(n^2)$ bit operations.

    Of course, instead of computing a bound $B$ this way, one could use a classical bound depending only on the coefficients of $R(w)$, like $B\leq 1+\max_i\left|\frac{r_i}{r_n} \right|$, as we will do in the proof of our complexity estimates to bound $\sigma$. In practice, however, such a bound is usually much less tight.    

    \subsection{Checking the Vanishing of $p$-Curvatures} \label{ssec:algopcurv}

    Our goal in this section is to explain how to efficiently check if the $p$-curvatures of equation \eqref{eq:deq1} vanish.
    We recall that the $p$-curvature of equation $y'(x)=u(x) y(x)$, is $u^{(p-1)}(x)+u(x)^p \bmod p$.
    Consider now $u(x)\in\Q(x)$ of degree $n\in\N$ and multiplicative height at most $H>0$.

    To compute $p$-curvatures we rely on an algorithm due to Bostan and Schost \cite{BS09} which is tailored specifically for first-order differential equations.
    By ``computing'' a $p$-curvature, we mean that we compute enough terms of its Taylor expansion in order to uniquely reconstruct it knowing a bound on the degrees of its numerator and denominator.
    Indeed, a rational function with numerator's and denominator's degrees at most $n$ is completely determined by $n$ and the first $2n$ coefficients of its Taylor series expansion.
    In~\cite[Theorem 2]{BS09}, Bostan and Schost compute the $p$-curvature starting from a differential equation over $\F_p(x)$, while we want as input a coefficient in $\Q(x)$, thus we cannot apply directly their result.
    
    \begin{prop}[Bostan, Schost]\label{prop:pcurvcomplexity}
        Let $u(x)\in\Q(x)$ of degree $n\in\N$, with coefficients bounded by $H>0$.
        For any prime number $p$ the computation of the $p$-curvature of $y'(x)=u(x) y(x)$ can be performed in $\tilde{O}(n^2\log(p)^2+n\log(H))$ bit operations, where we neglect factors $\log(n)$ and $\log(\log(p))$.
        
        The computation of all $p$-curvatures of $y'(x)=u(x) y(x)$ for $p\leq S$ can be done in $\tilde{O}(n^2S\log(H)/\log(S)+nS\log(S))$ bit operations. 
    \end{prop}

    \begin{proof}
        What we wish to do is first reduce $a(x)$ and $b(x)$ modulo $p$, and then compute the $p$-curvature.
        We need to perform $O(n)$ reductions modulo $p$, one for each coefficient of $a(x)$ and $b(x)$. With Euclidean divisions each of them can be done in $\tilde O(\log(p)+\log(H))$ bit operations.
        We then refer to \cite[Theorem 2]{BS09} to ensure that from this point the $p$-curvature can be computed in $\tilde O (n^2\log(p))$ operations in $\F_p$, thus $\tilde O(n^2 \log(p)^2)$ bit operations and we obtain the announced complexity for computing one $p$-curvature.
        The complexity of computing the $p$-curvatures for all primes up to $S$ follows from the Prime Number Theorem $\sum_{p\leq S}1 = O(S/\log(S))$ and the estimate $\sum_{p\leq S} \log(p)^2=O(S\log(S))$.
    \end{proof}

    \begin{rem}
        An efficient algorithm to compute the characteristic polynomials of the $p$-curvatures for any order differential equations for all primes $p\leq S$ simultaneously in $\tilde{O}(S(\log(H)+n)n^{\omega+1})$ bit operations is given by Pagès in \cite{Pag21}, where $\omega<2.3728596$ is an exponent for matrix multiplication.
        This is easily adapted to check the \emph{nilpotence} of the $p$-curvatures in the same complexity.
        Checking \emph{nullity} in such a complexity could prove crucial in the development of efficient algorithms for a potential effective version of the Grothendieck $p$-curvature conjecture for (classes of) higher order equations.
    \end{rem}

    \subsection{Complexity Estimate} \label{ssec:fullalgo}

    Our algorithm to decide the nature of all solutions of $y'(x)=u(x) y(x)$ by computing $p$-curvatures is presented above in Algorithm~\ref{algo:Honda}.

    \begin{prop} \label{prop:Delta2}
        We have \begin{equation} \label{eq:deltabound}
            \delta(\Delta_R)<\frac 4 {(\log2)^2}\left( (2n-1)^2\log(H)+4n^2\log(n+1) + \log\left(\frac{2^{4n^2-1}}{3^{n/2}(\pi n)^{2n-1}}\right) \right)^2.
        \end{equation} 
        In particular, with the notations as above, we have $\delta(\Delta_R)=\tilde O(n^4(\log H+\log n)^2)$.
    \end{prop}
    
    \begin{proof}
        The statement can be deduced by applying a bound on the size of coefficients of the factors of the Rothstein-Trager resultant, Proposition~\ref{prop:factheight}, and then Proposition~\ref{prop:RTheight} to compute a bound on $\Delta_R$: first we compute a bound on the height of $R(w)$, the square-free part of the Rothstein-Trager resultant of $a(x)$ and $b(x)$, and use a bound on the height of its Rothstein-Trager to obtain a bound for its leading coefficient, which is precisely $\Delta_R$. We conclude using Lemma~\ref{lem:delta}, below.
    \end{proof}

    % Our algorithm to decide the nature of all solutions of $y'(x)=u(x) y(x)$ by computing $p$-curvatures is presented above in Algorithm~\ref{algo:Honda}.

    % \begin{prop} \label{prop:Delta2}
    %     We have \begin{equation} \label{eq:deltabound}
    %         \delta(\Delta_R)<\frac 4 {(\log2)^2}\left( (2n-1)^2\log(H)+4n(n+1)\log(n+1) + \log\left(\frac{2^{4n^2-n}}{3^{n/2}\pi^n n^n}\right) \right)^2.
    %     \end{equation} 
    %     In particular, with the notations as above, we have $\delta(\Delta_R)=\tilde O(n^4(\log H)^2)$.
    % \end{prop}

    % \begin{proof}
    %     The statement can be deduced by applying a bound on the size of coefficients of the Rothstein-Trager resultant, Proposition~\ref{prop:RTheight}, twice to compute a bound on $\Delta_R$: first we compute a bound on the height of $R(w)$, that is $\max_k|r_k|$, where we use $\binom{n}{k}\leq \binom{n}{n/2}\leq \frac{2^{n+1/2}}{\sqrt{\pi n}}$, and iterate to obtain a bound for the leading coefficient of the Rothstein-Trager resultant of $R(w)$, which is precisely $\Delta_R$. We conclude using Lemma~\ref{lem:delta}, below.
    % \end{proof}

    \begin{lem} \label{lem:delta}
        Let $\Delta\in\N$ and $\delta = \prod_{p\vert\Delta}p^{1/(p-1)}$, then $ \delta(\Delta)\leq \frac{4}{\log(2)^2} \log(\Delta)^2$. 
    \end{lem}

    \begin{proof}
        First we remark that if $\ell\in\N$ is such that $\ell\# \leq \Delta < (\ell+1)\#$, then $\delta(\Delta)\leq \delta(\ell\#)$ where $\ell\# \coloneqq \prod_{p\leq \ell} p$ is the \emph{primorial} of $\ell$. This follows from the fact that $x^{1/(x-1)}$ is decreasing for $x\geq 2$. 
        We will now estimate $\delta(\ell\#)$.
        In \cite[Theorem 10]{RS62}, we find that for $\ell\geq 557$, we have $0.92 \cdot \ell\leq \log(\ell\#)$.
        Checking 
        \begin{equation}
            \ell\leq \frac{2}{\log(2)}\log(\ell\#) \label{eq:hashtag}
        \end{equation} numerically for $1\leq \ell\leq 557$ we conclude that $\frac \ell {\log(\ell\#)}$ attains its maximum for $\ell = 2$, i.e., \eqref{eq:hashtag} holds for all $\ell\geq 1$.
        Moreover, we have the estimate $\log(\delta(\ell\#))=\sum_{p\leq \ell} \frac{\log(p)}{p-1}\leq 2\sum_{p\leq \ell} \frac{\log(p)}{p} \leq 2\log(\ell)$ following \cite[Eq.~(2.5),~(2.11)]{RS62}, hence $\delta(\ell\#) \leq \ell^2$.
        Taking now $\ell$ such that $\ell\#\leq\Delta\leq (\ell+1)\#$ we obtain
        \begin{equation*}
            \delta(\Delta) \leq \delta(\ell\#) \leq \ell^2 \leq \left(\frac{2}{\log(2)} \log(\ell\#) \right)^2 \leq \frac{4}{\log(2)^2} \log(\Delta)^2. \qedhere
        \end{equation*}
    \end{proof}
    
    \begin{prop} \label{prop:complexity}
        Given $u(x)\coloneqq\frac{a(x)}{b(x)}\in\Q(x)$ of degree $n$ and multiplicative height $H$, deciding if all solutions of $y'(x)=u(x) y(x)$ are algebraic can be done using Algorithm~\ref{algo:Honda} by performing $\tilde O(\Delta_b^6B) = \tilde O(H^{12n-6}n^{12n}3^{-3n})$ bit operations, where $\tilde O$ hides factors polynomial in $n$ and logarithmic in $H$.
    \end{prop}

    \begin{proof}
        We compute $R(w), \Delta_b,$ and a bound on $\delta$ as described in the algorithm, and the maximal modulus of a root of $R(w)$, up to a constant factor by for instance setting $\tau=\frac 1 2$ in Proposition~\ref{prop:rootradius}. This we will use as an upper bound on $B$. With this we obtain an upper bound on $\sigma$. Clearly the complexities of these computations are negligible compared to the computation of sufficiently many $p$-curvatures as described below.

        By Theorem \ref{thm:intro2}, it suffices to compute the $p$-curvatures up to $\sigma$, defined in its statement, not dividing $\Delta_b$. Checking for a single prime $p$, whether $p$ divides $\Delta_b$ can be performed in $O(\log(\Delta_b)) $ bit operations, for all primes up to $\Delta_b$ na\"{\i}vely in $\tilde O(\Delta_b),$ where logarithmic factors in $\Delta_b$ are again neglected.
        
        We set $S=\sigma$ in Proposition \ref{prop:pcurvcomplexity}.
        In accordance with Proposition~\ref{prop:Delta2} the contribution of $\delta$ in the size of $\sigma$ is negligible and we have $\sigma=\tilde O(B\Delta_b^6)$, where we omit logarithmic factors in $\Delta_b$.
        Cauchy's bound on the roots of $R(w)= r_nw^n+\dots+r_0$ states that we can take $B$ at most $1+\max_i\left|\frac{r_i}{r_n} \right|$, hence $\sigma=\tilde O(\Delta_b^5H_R)$, where $H_R$ denotes the multiplicative height of $R(w)$.
        Using estimates from Proposition~\ref{prop:RTheight} to estimate this height and $\Delta_b$ we find both $\Delta_b$ and $H_R$ to be in $\tilde O(H^{2n-1}n^{2n-1/2}3^{-n/2})$.
        Taking the $6$-th power of this quantity and combining with Proposition~\ref{prop:pcurvcomplexity} concludes.
    \end{proof}

    We warn the reader that the multiplicative height $H$ appearing is exponential in the height $h$ appearing in other complexity estimates.
    
    This finishes the proof of the complexity estimates in Theorem~\ref{thm:intro2}. Note that Algorithm~\ref{algo:Honda} delays the computation of $B$ and $\sigma$, to after checking some amount of $p$-curvatures. This does not affect the complexity estimates presented here, but allows the algorithm to finish quickly for many examples, in which it returns transcendental. 
    
    Our algorithm is to be compared with the Algorithm~\ref{algo:ratroots} based on finding rational roots of a polynomial, whose complexity is polynomial in the degree and linear in the height as stated in Corollary~\ref{cor:ratroots}.
    This is vastly smaller than our exponential bounds in the degree $n$ and height $h$ to verify algebraicity.
    However, we expect rational functions yielding transcendental solutions to have $p$-curvatures that do not vanish for very small primes already, allowing Algorithm~\ref{algo:Honda} to conclude much faster than Algorithm~\ref{algo:ratroots}.
    Making an assumption as described above on the maximal expected first non-vanishing $p$-curvatures gives a heuristic about the differences in complexity in practice.
    This general behavior is highlighted on examples in the following section.

    \begin{rem}
        As highlighted already in Remark~\ref{rem:QbRTcomp}, our algorithm extends to the case of $u(x)\in\Qb(x)\setminus\Q(x)$, however the complexity estimate of Proposition~\ref{prop:complexity} does not hold because computations are more costly in number fields.
    \end{rem}

    \section{Implementations}\label{sec:implementations}

    \subsection{Algorithmic Improvements}

    The following paragraphs discuss our implementation of Algorithm \ref{algo:Honda} in \href{https://www.sagemath.org/}{SageMath}, that is available on GitHub with examples: \href{https://github.com/plucas0/Honda.git}{https://github.com/plucas0/Honda.git}.

    The crucial algorithmic idea of the design of Algorithm \ref{algo:Honda} is to delay the moment when costly computations are done, that is the computation of the Rothstein-Trager resultant $R(w)$, when the degree of the input increases.
    By checking the ``first'' $p$-curvatures, in our implementation for primes $p$ up to $\Delta_b$, we try to return ``Transcendental'' before having to compute $R(w)$.
    In case no $p$-curvature vanishes until this point, the computation of further $p$-curvatures, potentially up to $\sigma$, is inevitable and this has much higher complexity than the computation of $R(w)$.
    When computing the bound $B$ on the modulus of the roots of $R(w)$, the precision $\tau$ of the computation influences the complexity, see Proposition \ref{prop:rootradius}.
    In our case, computing all $p$-curvatures up to $\sigma$ already exceeds the capabilities of modern computers in most cases. For theoretical estimates, having $B$ multiplied by a constant factor does not change the complexity of the computations. Our implementation relies on native SageMath functions numerically computing zeroes of a polynomial, up to a fixed precision of $2^{-53}$.

    The na\"{\i}ve computation of $\delta(\Delta_R)$ can be costly, as it relies on the factoring of large integers. Using the bound obtained in \eqref{eq:deltabound} we avoid the computation of $\Delta_R$, an iterated resultant, and its prime factors. However, this bound appears to be highly non-optimal on examples. In our implementation we thus decide to compute $\Delta_R$ and try for a small number of primes (up to $T=1009$) to find its contributions to $\delta(\Delta_R)$. There remains $\Tilde \Delta_R$, the product of its large factors, whose contribution to $\delta(\Delta_R)$ is bounded from above by $T^{v/(T-1)}$ for $v=\log_T(\Tilde \Delta_R)$.

    Finally, the theoretical complexity of reducing polynomials modulo many primes $p$ simultaneously can be improved by adapting algorithms of polynomial multi-point evaluation for integers.
    However this is not done in our implementation, and reduction mod $p$ does not have a significant impact on computation time compared to the other operations performed by the algorithm.

    \begin{rem}
        The delaying of the computation of the full Rothstein-Trager resultant is not as easily possible if $u(x)\in\Qb(x)\setminus\Q(x)$ as explained in Remark \ref{rem:Qbalgo}, and we do not know a method to delay this computation.
        For this reason, our implementation only treats the case of $u(x)\in\Q(x)$.
    \end{rem}

    \subsection{Competing Algorithms} \label{ssec:competing}

    \subsubsection{Computation of $p$-curvatures}
    
    First of all we would like to ensure and to convince the reader that our implementation of the algorithm from \cite{BS09} to compute $p$-curvatures of first-order differential equations is indeed faster than Pagès' implementation of his algorithm \cite{Pag21} applied to first-order equations.
    The following timings in average were noted for both algorithms when asked to compute $p$-curvatures for primes up to $S$ on random polynomial inputs $a(x),b(x)\in\Z[x]$ of given multiplicative height $H$ and degrees $n-1$, respectively $n$.

    \begin{table}[ht]
      \centering
    \begin{tabular}{c|c|c||c|c}
        Degree & Height & $S$ & BS09 & Pag21 \\ \hline
        $3$ & $2$ & $100$ & $20$ ms & $126$ ms  \\ 
        $3$ & $2^{10}$ & $100$ & $23$ ms & $120$ ms \\ 
        $3$ & $2^{80}$ & $100$ & $25$ ms & $140$ ms \\ 
        $20$ & $2$ & $100$ & $0.08$ s & $13$ s \\ 
        $3$ & $2$ & $1000$ & $0.2$ s & $1.2$ s
    \end{tabular}
    \caption{Computation time of two algorithms computing $p$-curvatures.}
    \label{tab:pcurv}
    \end{table}
    We see that for ``small'' inputs Bostan and Schost's algorithm is faster than Pagès' algorithm, and this remains the case when increasing any parameter.

    \subsubsection{Maple's \texttt{istranscendental}}

    The maple package gfun \cite{SZ94} has a command called \texttt{istranscendental} based on the algorithms described in~\cite{BSS25}, aiming to prove transcendence of a given solution of a given differential equation. Among other things, it checks whether the singularities of the minimal differential operator annihilating a given function are all regular, and if its local exponents are rational. This is a necessary criterion for the algebraicity of the solution, so in case one of these conditions is violated, the command outputs \texttt{true}. For first-order differential equations however, Proposition~\ref{prop:TFAE0a} ensures that the conditions are also sufficient for algebraicity. While \texttt{istranscendental} outputs \texttt{FAIL} in this case, we can actually conclude algebraicity and thus we are (ab)using the command as an implementation of Algorithm~\ref{algo:indeq}.

    \subsubsection{Rational Roots of the Rothstein-Trager resultant}

    In Section \ref{ssec:factoringcomp} we investigated the theoretical complexity of Algorithm~\ref{algo:ratroots} which decides algebraicity of the solutions of~\eqref{eq:deq1} by checking whether the Rothstein-Trager resultant of the coefficient $u(x)$ completely factors over $\Q$. This is equivalent to checking that the number of its rational roots is equal to its degree, as the algorithm describes. For this purpose, SageMath's native \texttt{roots} command performs a full factorization in $\Q[x]$, while Maple has more efficient methods implemented. We will see in the timings in the next section, that the computation of the resultant is the computationally more complex task in practice.

    \subsection{Examples}

    \subsubsection{Small Inputs} \label{sssec:small}

    Tests of our implementation of Algorithm \ref{algo:Honda} on very small examples, with small inputs and algebraic output, quickly reach the limits of modern computers' capabilities, as displayed in Table~\ref{tab:small}.
    We also compare our algorithm ($p$-curv) with \texttt{istranscendental} (ist) and the computation of the Rothstein-Trager resultant and its rational roots (RR) in Maple, and with computation of the resultant and performing its full factorization (fact) in SageMath.

    \begin{table}[ht]
      \centering
    \begin{tabular}{c||c|c|c||c|c|c}
        $\cfrac{a(x)}{b(x)}$ & $\sigma$ & Output & $p$-curv & ist & fact & RR \\ \hline \hline
        $\cfrac{3x-4}{2x^2-6x+4}$ & $265$ & A & $156$ ms & $45$ ms & $<1$ ms & $25$ ms  \\ \hline \hline
        $\cfrac{7x^2-3x-4}{2x^3+4x^2-6x+4}$ & $\approx 3\cdot 10^{29}$ & T & $5$ ms & $38$ ms & $<1$ ms & $30$ ms \\ \hline \hline
        $\cfrac{2x+1}{x^2+x+1}$ & $3851774$ & A & $14$min $44$s & $19$ ms & $<1$ ms & $24$ ms \\ \hline \hline
        $\cfrac{1}{x^2 - 4}$ & $\approx 10^{11}$ & A & DNF & $15$ ms & $<1$ ms & $22$ ms
    \end{tabular}
    \caption{Output and computation time of our implementation of Algorithm~\ref{algo:Honda} and timings of competing algorithms on a few examples. Here the output A stands for ``Algebraic'', and T stands for ``Transcendental.''}
    \label{tab:small}
    \end{table}

    In particular, on the last line, it takes approximately $1$ minute to compute $p$-curvatures for primes $p$ up to $2\cdot 10^5$, an extrapolation indicates it would take more than $2$ years to compute them all up to $\sigma\approx 10^{11}$. At the same time it is obvious that the polynomial $x^2-4$ splits in $\Q[x]$, hence we would like to return ``Algebraic'' instantly.

    \subsubsection{Large Random Inputs}

    Certifying algebraicity using our approach is difficult. Both the theoretical complexity and the timings of our algorithm suggest so, already for small degree and height of the input. Proving transcendence, however, is in general much easier as we expect that in this case -- except for very specific polynomials -- the $p$-curvatures for some small primes $p$ will not vanish.
    This might seem counterintuitive when compared to numbers -- for which proving transcendence is considered hard -- but this behavior was already observed in the functional context in \cite{BSS25}.
    By ``small'', we mean that we expect that it suffices to check a number of $p$-curvatures significantly smaller than the bound $\sigma$ of Theorem~\ref{thm:intro2}, but the latter is, to our knowledge, the best proven bound.
    Thus, in practice, our algorithm should be viewed rather only as a semi-algorithm in the sense that if the output is algebraic the algorithm will not terminate in reasonable time, except in very specific examples.
    The observed behavior worsens as the input degree and height increase, hence let us focus on the case where the output is transcendental.
    Note that testing $p$-curvatures up to some (arbitrarily chosen) bound and returning ``Transcendental'' when one does not vanish is possible without relying on our Theorem~\ref{thm:intro2}, one only needs to know the ``bad primes'', cf. Proposition~\ref{prop:Put}.

    It is expected that the solution of a first order differential equation with coefficient $u(x)=a(x)/b(x)$ for integer polynomials $a(x), b(x)$ of fixed degree $n\in\N_{\geq 1}$ with random integer coefficients in $[-H,H]$, for $H\in\N_{>0}$ is transcendental, and checking that is possible by computation of very few $p$-curvatures.
    In Table~\ref{tab:random}, computation times of all algorithms mentioned in Section~\ref{ssec:competing}, including Algorithm~\ref{algo:Honda}, are compared on inputs consisting of two such polynomials, with the first polynomial's degree strictly smaller than the second's.
    The algorithms we compare are our implementation of Algorithm~\ref{algo:Honda} in SageMath ($p$-curv), Maple's \texttt{istranscendental} (ist), computing the Rothstein-Trager resultant (RT) and using Maple's \texttt{roots} command for finding the rational roots of this resultant (RR), and computing the Rothstein-Trager resultant in SageMath and factoring it (fact). Of course the last algorithm is expected not to be competitive, still we display the timings to have a comparison in SageMath, the same system as our implementation.

    \begin{table}[ht]
        \centering
        \begin{tabular}{c|c||c|c|c|c|c}
             &  &  &  & \multicolumn{2}{c|}{RT+RR (Maple)} &  \\ 
            Degree & Height & $p$-curv & ist & RT & RT+RR & fact (Sage) \\\hline \hline
            $10$ & $2^{10}$ & $1$ ms & $12$ ms & $3$ ms & $3$ ms & $<1$ ms \\
            $20$ & $2^{10}$ & $2$ ms & $24$ ms & $9$ ms & $10$ ms & $4$ ms \\
            $20$ & $2^{20}$ & $2$ ms & $25$ ms & $19$ ms & $21$ ms & $7$ ms \\
            $40$ & $2^{10}$ & $4$ ms & $71$ ms & $46$ ms & $49$ ms & $79$ ms \\
            $40$ & $2^{20}$ & $5$ ms & $76$ ms & $100$ ms & $107$ ms & $171$ ms \\
            $80$ & $2^{10}$ & $0.1$ s & $0.3$ s & $0.3$ s & $0.3$ s & $2.4$ s \\
            $80$ & $2^{20}$ & $0.1$ s & $0.3$ s & $0.6$ s & $0.6$ s & $5.0$ s \\
            $160$ & $2^{10}$ & $0.4$ s & $1.8$ s & $2.4$ s & $2.4$ s & $83$ s \\
            $160$ & $2^{20}$ & $0.4$ s & $1.9$ s & $3.9$ s & $4.0$ s & $182$ s
        \end{tabular}
        \caption{Average computation time of various algorithms on random rational function inputs of prescribed degree and height.}
        \label{tab:random}
    \end{table}

    The timings are given for random inputs $a(x),b(x)\in\Z[x]$ of multiplicative height at most $H$ and degree $n-1$, respectively, $n$, where Algorithm~\ref{algo:Honda} returns ``Transcendental.''
    The tests for degrees up to $40$ were performed on samples of hundreds of random rational functions, the size of the sample for the tests in degree $80$ and $160$ were respectively $100$ and $10$.
    In all our experiments, we observed that most examples return transcendental after computing a non-vanishing $p$-curvature with $p$ no greater than $17$. In such cases, we avoid the computation of the Rothstein-Trager polynomial and only compute its leading coefficients, $\res_x(b(x),b'(x))$, whereas the factorization approach cannot take this shortcut.

    The timings of computing Rothstein-Trager resultants separately were only performed in Maple, which outperforms SageMath on this task. We observe that the time for finding all rational roots of the resultant is negligible compared to computing the resultant in the first place, and that our implementation outperforms any algorithm requiring the computation of the resultant. The timings nicely illustrate, that for generic examples our algorithm returns ``Transcendental'' quickly, as claimed at the end of Section~\ref{ssec:fullalgo}.

    One can hand-craft examples on which we need to test relatively large primes $p$ -- at least compared to the generic behavior. 
    In the case of a quadratic polynomial, this smallest prime for which it does not factor into a product of two linear factors is the smallest quadratic non-residue in its splitting field.
    Examples of quadratic fields with prescribed least quadratic non-residue can be found in \cite{MGT21}. From this one can construct coefficients $u(x)$ of differential equations, for which the first prime for which the $p$-curvature does not vanish is comparatively large.
    For instance, if $a(x)=1$, $b(x)=x^2-3818929$, the prime $p=2$ divides $\Delta_b =\res(b(x),b'(x))$, and for all other primes up to $43$ the $p$-curvatures vanish, but not the $47$-curvature. Our algorithm runs in $80$ milliseconds on this example.

    More generally, we already discussed the connection of our problem with bounding the least prime that does not split in a number field in Section \ref{sec:comparison}.

    A systematic construction of polynomials $R(w)$ with ``large'' smallest prime $p$ for which the reduction of $R(w)$ does not split completely, or of coefficients $u(x)$ of differential equations for which the first non-vanishing $p$-curvature is high, is not known to us.

	\sloppy
	\printbibliography

	\textsc{Faculty of Mathematics, University of Vienna, Oskar-Morgenstern-Platz 1, 1090, Vienna, Austria}
	
	\textit{Email: }\href{mailto:florian.fuernsinn@univie.ac.at}{\texttt{florian.fuernsinn@univie.ac.at}}\medskip
    
    \textsc{Université Paris-Saclay, UVSQ, CNRS, UMR-8100, Laboratoire de Mathématiques de Versailles, 78000, Versailles, France.}
	
	\textit{Email: }\href{mailto:lucas.pannier@uvsq.fr}{\texttt{lucas.pannier@uvsq.fr}}
\end{document}